\documentclass[centertags,12pt]{amsart}
\usepackage{amssymb}
\usepackage{verbatim}
\usepackage{array}
\usepackage{latexsym}
\usepackage{enumerate}
\usepackage{amsmath}
\usepackage{amsfonts}
\usepackage{amsthm}
\usepackage{color}
\usepackage[english]{babel}

\usepackage{tikz}
\usepackage{wasysym}
\usepackage[foot]{amsaddr}

\usepackage{mathrsfs}
\usepackage{setspace}
\usepackage{ textcomp }
\usepackage{graphicx}
\usepackage{cite}
\usepackage{comment,cancel}
\graphicspath{ {images/} }
\usepackage{listings}

\newtheorem{theorem}{Theorem}[section]
\newtheorem{proposition}[theorem]{Proposition}
\newtheorem{lemma}[theorem]{Lemma}
\newtheorem{corollary}[theorem]{Corollary}

\newtheorem{example}[theorem]{Example}

\def\cH{\mathcal H}

\def\Aut{\mbox{\rm Aut}}
\def\K{\mathbb{K}}

\def\deg{\mbox{\rm deg}}

\def\Aut{\mbox{\rm Aut}}

\newcommand{\Fix}{\mbox{\rm Fix}}

\newcommand{\PGU}{\mbox{\rm PGU}}

\newcommand{\SL}{\mbox{\rm SL}}

\newcommand{\aut}{\mbox{\rm Aut}}

\def\neg1{\text{\boldmath$1$}}

\def\neg1{\text{\boldmath$1$}}

\def\cH{\mathcal H}

\def\FF{\mathbb{F}}

\def\Fqq{{\mathbb{F}_{q^2}}}

\def\Fqn{\mathbb{F}_{q}^{n}}

\newtheorem{remark}[theorem]{Remark}

\begin{document}

\title{Non-isomorphic subfields of the BM and GGS maximal function fields}
\author{Peter Beelen}
\address{Department of Applied Mathematics and Computer Science, Technical University of Denmark, DK-2800, Kongens Lyngby, Denmark}
\email{pabe@dtu.dk}

\author{Tobias Drue}
\email{tobiasd@live.dk}

\author{Maria Montanucci}
\address{Department of Applied Mathematics and Computer Science, Technical University of Denmark, DK-2800, Kongens Lyngby, Denmark}
\email{marimo@dtu.dk}

\author{Giovanni Zini}
\address{Universitá degli Studi di Modena e Reggio Emilia, Dipartimento di Scienze Fisiche, Informatiche e Matematiche, Via Campi, 213/b Modena MO, Italy}
\email{giovanni.zini@unimore.it}


\begin{abstract}
In 2016 Tafazolian et al. \cite{TTT} introduced new families of $\mathbb{F}_{q^{2n}}$-maximal function fields $\mathcal{Y}_{n,s}$ and $\mathcal{X}_{n,s,a,b}$ arising as subfields of the first generalized GK function field (GGS). In this way the authors found new examples of maximal function fields that are not isomorphic to subfields of the Hermitian function field. 
In this paper we construct analogous function fields $\tilde{\mathcal{Y}}_{n,s}$ and $\tilde{\mathcal{X}}_{n,s,a,b}$ as subfields of the second generalized GK function field (BM) and determine their automorphism groups. Using that the automorphism group is an invariant under isomorphism, we show that the function fields $\tilde{\mathcal{Y}}_{n,s}$ and ${\mathcal{Y}}_{n,s}$, as well as  $\tilde{\mathcal{X}}_{n,s,a,b}$ and $\mathcal{X}_{n,s,a,b}$, are not isomorphic unless $m/s$ divides $q^2-q+1$ and $3$ divides $n$. In other words, the difference between the BM and GGS function fields can be found again at the level of the subfields that we consider.
\end{abstract}

\maketitle

\section{introduction}


A function field $F$ of genus $g$ defined over a finite field $\mathbb{F}_{q^2}$ with square cardinality $q^2$ is called maximal, if the Hasse–Weil upper bound is attained. More precisely, the Hasse–Weil bound states that the number $N(F)$ of places of degree one of $F$ satisfies $N(F)  \leq  q^2 + 1 + 2gq$, and $F$ is called maximal over $\mathbb{F}_{q^2}$ if $N(F) = q^2+1+2gq$. Nowadays a lot of research is available in the literature about maximal function fields, or equivalently, from a more geometric angle, about maximal algebraic curves. Apart for their intrinsic theoretical interest as extremal objects, maximal function fields have attracted a lot of attention in recent decades due to their applications to coding theory and cryptography. Such function fields are indeed special for the structure of the so-called Weiestrass semigroup at a rational place, which
is the main ingredient used in the literature to construct AG codes with good parameters.

The most important and well-studied example of a $\mathbb{F}_{q^2}$-maximal function field is the Hermitian
one $\mathcal{H}_q:= \mathbb{F}_{q^2}(x, y)$ with $y^{q+1} = x^q + x$.
A reason for the importance of $\mathcal{H}_q$ is that it has the largest possible genus $q(q - 1)/2$ that an $\mathbb{F}_{q^2}$-maximal function field can have. A result commonly attributed to Serre,
see \cite[Proposition 6]{KS}, gives that any subfield of a maximal function field is still maximal.
Therefore many $\mathbb{F}_{q^2}$-maximal function fields have been obtained as subfields of already known maximal function fields, in particular of the Hermitian function field $\mathcal{H}_q$. For a while it was speculated in the research community that perhaps all maximal function fields could be obtained in this way up to isomorphism. However, it was shown by Giulietti and Korchmáros in \cite{GK} that this is not the case. 

Giulietti and Korchmáros constructed indeed an maximal function field over $\mathbb{F}_{q^6}$, today referred to as the GK function field, which is not isomorphic to any subfield of the Hermitian function field $\mathcal{H}_{q^3}$ whenever $q$ is larger than $2$. 
Garcia, G\"uneri, and Stichtenoth presented in \cite{GGS} a new family of maximal function fields $\mathbb{F}_{q^{2n}}$ for any odd $n$, known as the GGS function fields, which generalize the GK ones (they coincide when $n=3$) and are not isomorphic to any Galois subfield of the Hermitian
function field $\mathcal{H}_{q^n}$, see \cite{DM,GMZ}.
Another generalization of the GK function field over $\mathbb{F}_{q^{2n}}$ for any odd $n$ has been introduced by
Beelen and Montanucci in \cite{BM}, which is now known as the BM function field. It is not
isomorphic to any Galois subfield of the Hermitian function field $\mathcal{H}_{q^n}$ as well, unless $q=2$. The GGS and the BM function fields are very similar, in the sense that they both coincide with the GK function field for $n=3$, and they have the same genus for any $n$. However, the GGS and the BM function fields are not isomorphic, even over the algebraic closure of $\mathbb{F}_{q^{2n}}$, as they have different automorphism groups; see \cite{BM}.

Tafazolian, Teherán-Herrera, and Torres presented in \cite{TTT} two further examples $\mathcal{Y}_{n,s}$ and $\mathcal{X}_{a,b,n,s}$ of maximal
function fields over $\mathbb{F}_{q^{2n}}$ for any odd $n$, that are not subfields of the Hermitian function field $\mathcal{H}_{q^n}$. They are closely related to the GGS function fields mentioned above. In fact, $\mathcal{Y}_{n,s}$ and $\mathcal{X}_{a,b,n,s}$ are Galois subfield of the GGS function field over $\mathbb{F}_{q^{2n}}$; in particular, $\mathcal{Y}_{n,1}$ is the GGS function field and hence $\mathcal{Y}_{3,1}$ is the GK function field over $\mathbb{F}_{q^6}$. 
Considering the hard problem of deciding whether two maximal function fields are isomorphic or not, it is then natural to ask: \textit{is it possible to find Galois subfields $\tilde{\mathcal{X}}_{a,b,n,s},\tilde{\mathcal{Y}}_{n,s}$ of the BM function field, replicating the similarity and difference between the GGS and the BM function fields? That is, can we find Galois subfields $\tilde{\mathcal{X}}_{a,b,n,s}$ and $\tilde{\mathcal{Y}}_{n,s}$ of the BM function field with the same genus of  $\mathcal{X}_{a,b,n,s}$ and $\mathcal{Y}_{n,s}$ but with different automorphism groups?}
In this paper, which is based on the Master thesis \cite{Tobias} of the second author, we give a positive answer to this question, constructing two new families of maximal function fields and determining explicitly their full automorphism groups.

The paper is organized as follows. In Section \ref{sec:prel} we provide some preliminary results about maximal function fields and their automorphism groups. In Section \ref{sec:Y} we define the maximal function field $\tilde{\mathcal{Y}}_{n,s}$ and compute its automorphism group. Section \ref{sec:X} introduces the function field $\tilde{\mathcal{X}}_{a,b,n,s}$, studies some Weierstrass semigroups at places of $\tilde{\mathcal{X}}_{a,b,n,s}$, and determines its full automorphism group. Finally, Section \ref{sec:covering} provides a discussion about the isomorphism problem with respect to Galois subfields of the Hermitian function fields for $\tilde{\mathcal{Y}}_{n,s}$ and $\tilde{\mathcal{X}}_{a,b,n,s}$.

\section{Preliminary results}\label{sec:prel}

\subsection{Automorphism groups of algebraic function fields}
In this section we collect some results about automorphism groups of function fields that we are going to use later.

Let $\mathbb{F}$ be a field of characteristic $p\geq0$, $\mathbb{K}$ be the algebraic closure of $\mathbb{F}$, and $F$ be an algebraic function field of genus $g \ge 2$ defined $\mathbb{F}$.
We denote by $\aut(F)$ the group of all $\mathbb{K}$-automorphisms of the function field $\K F$ obtained from $F$ by extending its constant field to $\K$. 
The assumption $g\geq 2$ ensures that $\aut(F)$ is finite. However the classical Hurwitz bound $|\aut(F)| \leq 84(g(F)-1)$ over the complex numbers fails in positive characteristic, and there exist four families of function fields satisfying $|\aut(F)|\geq 8g^3$; see \cite{stichtenoth1973II}, \cite{henn1978}, and \cite[Section 11.12]{HKT}. More precisely, the following is known.
\begin{theorem}\label{thm:stichtenoth_henn}
Let $F$ be a function field of genus $g \ge 2$ defined over an algebraically closed field $\K$ of characteristic $p$.
Then $|\aut(F)|<8g^3$ unless one of the following cases occurs up to $\mathbb{K}$-isomorphism:
\begin{enumerate}
\item $p = 2$ and $F=\K(x,y)$ with $y^2 + y = x^{2k+1}$, where $k > 1$; 
\item $p > 2$ and $F=\K(x,y)$ with $y^2 = x^{q} -x$, where $q=p^h$ and $h > 0$;
\item $F$ is the Hermitian function field $\K(x,y)$ with $y^{q+1}=x^q+x$, where $q = p^h$ and $h > 0$;
\item $p=2$ and $F$ is the Suzuki function field $\K(x,y)$ with $x^{q_0}(x^q + x) = y^q + y$, where $q_0 = 2^r$, $r \geq 1$ and $q = 2q_0^2$. 
\end{enumerate}
\end{theorem}

For a subgroup $G$ of $\aut(F)$, let $\bar F={\rm Fix}(G)$ be the fixed field of $G$. The field extension $F/\bar F$ is Galois of degree $|G|$.
A place $P\in F$ is a ramification place of $G$ if the stabilizer $G_P$ of $P$ in $G$ is nontrivial; the ramification index $e_P$ being $|G_P|$. The ramification locus of $G$ is the set of all ramification places. The $G$-orbit of a place $P$ is the subset
$o=\{g(P)\,\mid\, g\in G\}$ of the set of places of $F$. The orbit $o$ is called {\em long} if $|o|=|G|$, otherwise it is called {\em short}. For a place $\bar{Q}$ in $\bar F$, the $G$-orbit $o$ lying over $\bar{Q}$ consists of all places $P$ in $F$ such that $P|\bar{Q}$. If $P\in o$ then $|o|=|G|/|G_P|$ and hence $P$ is a ramification place if and only if $o$ is a short $G$-orbit. It can happen that $G$ has no short orbits, that is, every non-trivial element in $G$ acts fixed-point-free on the places of $F$; in this case we say that the function field extension $F / \bar F$ is unramified. On the other hand, $G$ has always finitely many short orbits.

For a non-negative integer $i$, the $i$-th ramification group $G_P^{(i)}$ of $F$
at $P$ is defined as
$$G_P^{(i)}=\{\alpha \in G_P \mid v_P(\alpha(t)-t)\geq i+1\}, $$ where $t\in F$ is a local parameter at
$P$. Not that $G_P^{(0)}=G_P$. In some texts the notation $G_i(P)$ is used instead of $G_P^{(i)}$, e.g. in \cite[Chapter
IV]{serre1979}. 
The subgroup chain $G_P \supseteq G_P^{(1)} \supseteq \ldots \supseteq G_P^{(i)} \supseteq  \ldots$ is described in the following lemma.

\begin{lemma}{\cite[Theorem 11.49 and Lemma 11.74]{HKT}}  \label{highnormal}
\begin{enumerate}
\item $G_P^{(0)}=G_P=S \rtimes H$, where $S$ is a $p$-group and $H$ is cyclic of order not divisible by $p$.
\item $G_P^{(1)}=S$ is the unique Sylow $p$-subgroup (and maximal normal $p$-subgroup) of $G_P$.
\item For $i \geq 1$, $G_P^{(i)}$ is normal in $G_P$ and the quotient $G_P^{(i)}/G_P^{(i+1)}$ is elementary abelian.
\end{enumerate}
\end{lemma}

Let $\bar{g}$ be the genus of $\bar F={\rm Fix}(G)$. The \emph{Hurwitz
genus formula} \cite[Theorem 3.4.13]{stichtenoth} gives
    \begin{equation}
    \label{eq1}
2g-2=|G|(2\bar{g}-2)+\sum_{P} d_P,
    \end{equation}
    where the different exponent $d_P$ at a place $P$ of $F$ is given by Hilbert's different formula
\begin{equation}
\label{eq1bis}
d_P= \sum_{i\geq 0}(|G_P^{(i)}|-1),
\end{equation}
see \cite[Theorem 11.70]{HKT}. By Hilbert's different formula, $\sum_{P} d_P$ becomes a double summation and hence can also be computed by summing over all non-trivial automorphisms $\alpha$ of $G$ the contribution $i(\sigma)$ of $\sigma$ to $\sum_{P} d_P$, that is the number of ramification groups containing $\alpha$.


A subgroup of $\aut(F)$ is  a prime-to-$p$ group (or a $p'$-subgroup) if its order is prime to $p$. A subgroup $G\leq\aut(F)$ is {\em{tame}} if the stabilizer $G_P$ in $G$ 
of any place $P$ is a $p'$-group. Otherwise, $G$ is called {\em{non-tame}} (or {\em{wild}}). By \cite[Theorem 11.56]{HKT}, if $|G|>84(g-1)$ then $G$ is non-tame. An orbit $o$ of $G$ is {\em{tame}} if $G_P$ is a $p'$-group for $P\in o$. 
Specializing to maximal function fields, more can be said about its automorphisms. 
 
\begin{lemma}{\rm (\cite[Prop. 3.8, Thm. 3.10]{GSY})} \label{prankpel} Let $F$ be an $\mathbb{F}_{q^2}$-maximal function field of genus $g\geq2.$ Every automorphism of $F$ over the algebraic closure of $\mathbb{F}_{q^2}$ is defined over $\mathbb{F}_{q^2}.$ In particular, the automorphism group $\aut(F)$ fixes the set of $\mathbb{F}_{q^2}$-rational places of $F$.  
\end{lemma}

Lemma \ref{prankpel} can be used to ensure that a Sylow $p$-subgroup of a non-tame 
automorphism group of an $\mathbb{F}_{q^2}$-maximal function field $F$ fixes exactly one $\mathbb{F}_{q^2}$-rational place of $F$.
  
\begin{corollary}\label{1actionp2} Let $p>0$ be the characteristic of $\mathbb{F}_{q^2}$ and let $F$ be an $\mathbb{F}_{q^2}$-maximal function field of genus $g \geq 2$ such that $p \mid |\aut(F)|$. If $H$ is a $p$-subgroup of 
$\aut(F)$, then $H$ fixes exactly one place $P$ and $P$ is $\mathbb{F}_{q^2}$-rational. The subgroup $H$ acts semiregularly on the 
set of the remaining $\mathbb{F}_{q^2}$-rational places of $F$. 
Moreover, if $S_1$ and $S_2$ are distinct Sylow $p$-subgroups of $\aut(F)$ with fixed places $P_1$ and $P_2$ respectively, then $P_1\ne P_2$ and $S_1,S_2$ intersect trivially.
\end{corollary}

\begin{proof} Let $H$ be a $p$-subgroup of $\aut(F)$. By Lemma \ref{prankpel},  $H$ acts on the set $\mathcal{P}$ of $\mathbb{F}_{q^2}$-rational places of $F$. Since $F$ is $\mathbb{F}_{q^2}$-maximal, the size of $\mathcal{P}$ is congruent to $1$ modulo $p$. Hence $H$ fixes at least one place $P \in \mathcal{P}$.
On the other hand, being maximal, $F$ has $p$-rank zero. Then by \cite[Lemma 11.129]{HKT} any nontrivial element in $H$ fixes exactly one place of $F$.
Thus, $H$ itself fixes exactly one place $P \in \mathcal{P}$ and acts semiregularly on the remaining $\mathbb{F}_{q^2}$-rational places of $F$.
Let $S_1,S_2,P_1,P_2$ be as in the hypothesis. As $S_1,S_2$ are conjugate under $\aut(F)$, take $\alpha\in\aut(F)$ with $\alpha^{-1}S_2\alpha=S_1$. Then $\alpha(P_1)=P_2$. Suppose by contradiction $P_1=P_2$. Then $\alpha$ stabilizes $P_1$, and hence normalizes the Sylow $p$-subgroup $S_1$ of $\aut(F)_{P_1}$. Then $S_2=S_1$, a contradiction. From $P_1\ne P_2$ and \cite[Lemma 11.129]{HKT} it follows also $S_1\cap S_2=\{id\}$.
\end{proof}

If the bound $|G|>84(g(F)-1)$ is satisfied a lot can be said about the structure of the short orbits of $G$ on the places of $F$. The following theorem lists indeed all the possibilities.

\begin{theorem} \rm{\cite[Theorem 11.56 and Lemma 11.111]{HKT}} \label{Hurwitz}
 Let $F$ be an algebraic function field of genus $g \geq  2$ defined over a field $\K$ of characteristic $p$. Suppose that $G\leq\aut(F)$. Then:

\noindent
$(i)$ If $G$ has at least five short orbits, then $|G| \leq 4(g-1)$. 

\noindent
$(ii)$ If $G$ has four short orbits, then $|G| \leq 12(g-1)$.

\noindent
$(iii)$ If $G$ has exactly one short orbit, then the length of this orbit divides $2g-2$.

\noindent
$(iv)$ If $|G|>84(g-1)$, then $p>0$, ${\rm Fix}(G)$ is rational, and $G$ has at most three short orbits: 
\begin{enumerate}
\item exactly three short orbits, two tame of length $|G|/2$ and one non-tame, with $p \geq 3$;
\item or exactly two short orbits, both non-tame; 
\item or only one short orbit which is non-tame, whose length divides $2g-2$; 
\item or exactly two short orbits, one tame and one non-tame.
\end{enumerate}
\end{theorem}

A useful tool that we will use is the so-called Weierstrass semigroup $H(P)$ at a place $P$:
$$H(P):=\{i \in \mathbb{N} \mid \exists f \in F, (f)_\infty=iP \}.$$
It is well-known that $H(P)$ is a numerical semigroup. The Weierstrass gap theorem implies that the set of gaps $G(P):=\mathbb{N} \setminus H(P)$ of $P$ has cardinality $g$; see e.g. \cite[Theorem 1.6.8]{stichtenoth}. The following lemma provides a technique to compute gaps.

\begin{lemma}{\cite[Corollary 14.2.5]{VillaSalvador}} \label{gapsdiff}
 Let $F$ be an algebraic function field of genus $g$ defined over $\mathbb{K}$, $P$ be a rational place of $F$, and $n$ be a non-negative integer.
 Then $n$ is a gap at $P$ if and only if there exists a differential $\omega$ on $F$ such that $\omega$ is regular (that is, the divisor of $F$ is an effective divisor) and $n=v_P(\omega)+1$.
\end{lemma}
Places in the same $\aut(F)$-orbit have the same Weierstrass semigroup (\cite[Lemma 3.5.2]{stichtenoth}).

\subsection{Automorphisms of the Hermitian function field}\label{subsec:Hermitian}
Denote by $\mathcal{H}_q:=\mathbb{F}_{q^{2}}(x,y)$ the Hermitian function field over $\mathbb{F}_{q^{2}}$ where $y^{q+1}-x^{q+1}+1=0$. The functions $x$ and $y$ have $q+1$ distinct poles which we denote by $P_{\infty}^j$ with $j=1,\dots,q+1$. The other rational places are in one-to-one correspondence with the pairs $(a,b)$, where $a,b \in \FF_{q^2}$ satisfy $b^{q+1}-a^{q+1}+1=0$; we denote by $P_{(a,b)}$ the common zero of $x-a$ and $y-b$.
From \cite{MZ2016}, the group $\aut(\mathcal{H}_q)$ has a maximal subgroup $M_\ell$ of order $q(q-1)(q+1)^2$ which is isomorphic to $S_\ell \rtimes C_{q+1}$, where $C_{q+1}$ is a cyclic group of order $q+1$ and $S_\ell$ is the commutator subgroup of $M_\ell$, satisfying $S_\ell \cong\SL(2,q)$. An explicit description of $M_\ell$ and $S_\ell$ is given in \cite{MZcomm} in matrix representation:
$$M_\ell=\Bigg\{ \begin{pmatrix} a & b & 0 \\ c & d & 0 \\ 0 & 0 & 1\end{pmatrix} \mid a,b,c,d \in \mathbb{F}_{q^2}, \  d^{q+1}-b^{q+1}=1, \ a^{q+1}-c^{q+1}=1, -a^qb+c^qd=0 \Bigg \},$$
$$S_\ell=\Bigg\{ \begin{pmatrix} a & b & 0 \\ c & d & 0 \\ 0 & 0 & 1\end{pmatrix} \in M_\ell \mid ad-bc=1 \Bigg\} \cong \SL(2,q).$$

We will denote by $\alpha_{a,b,c,d}$ the automorphism of $\mathcal{H}_q$ associated to the matrix in $M_\ell$ with entries $a,b,c,d$; that is, $\alpha_{a,b,c,d}$ is determined by the equations
\[ \alpha_{a,b,c,d}(x)=ax+by, \qquad \alpha_{a,b,c,d}(y)=cx+dy.\]
Note that a direct computation gives, for any $\alpha_{a,b,c,d} \in M_\ell$, 
$$\alpha_{a,b,c,d}(y)^{q+1}-\alpha_{a,b,c,d}(x)^{q+1}+1=y^{q+1}-x^{q+1}+1,$$
so that $\alpha_{a,b,c,d}$ is indeed an automorphism of $\mathcal{H}_q=\mathbb{F}_{q^{2}}(x,y)$.
For any $\alpha_{a,b,c,d} \in M_\ell$, we have
\begin{eqnarray*}
c^{q(q+1)} & = & c^{q(q+1)}(d^{q+1}-b^{q+1}) =  (c^qd)^{q+1}-(bc^q)^{q+1}\\
& = & (a^qb)^{q+1}-(bc^q)^{q+1} =  b^{q+1}(a^{q+1}-c^{q+1})^q =  b^{q+1} 
\end{eqnarray*}
and similarly $a^{q(q+1)}=d^{q+1}$.
Together with $a^qb=c^qd$, this means that for any $\alpha_{a,b,c,d} \in M_\ell$ there exists $\zeta \in \mathbb{F}_{q^2}$ such that $\zeta^{q+1}=1$, $b=\zeta c^q$, and $d=\zeta a^q.$ In particular, $b=c^q$ and $a=d^q$ for every $\alpha_{a,b,c,d} \in S_\ell.$ Thus,
\begin{equation}\label{eq:Ml}
M_\ell=\Bigg\{ \begin{pmatrix} a & \zeta c^q & 0 \\ c & \zeta a^q & 0 \\ 0 & 0 & 1\end{pmatrix} : a,c \in \mathbb{F}_{q^2}, \  a^{q+1}-c^{q+1}=1, \zeta^{q+1}=1 \Bigg \},
\end{equation}
\begin{equation}\label{eq:Sl}
S_\ell=\Bigg\{ \begin{pmatrix} a & c^q & 0 \\ c & a^q & 0 \\ 0 & 0 & 1\end{pmatrix} : a,c \in \mathbb{F}_{q^2}, \  a^{q+1}-c^{q+1}=1 \Bigg\}.
\end{equation}

The group $S_\ell$ has $q(q^2-1)$ elements and hence any of its Sylow $p$-subgroups has order $q$. One possible $p$-Sylow subgroup of $S_\ell$ is the following:
\begin{equation}\label{eq:Eq}
E_q=\Bigg\{ \begin{pmatrix} a & a-1 & 0 \\ -a+1 & -a+2 & 0 \\ 0 & 0 & 1\end{pmatrix} : a \in \mathbb{F}_{q^2}, \  a^{q}+a=2 \Bigg\}.
\end{equation}
Indeed, it is not hard to verify that $E_q$ is elementary abelian, of order $q$, and contained in $S_\ell$; it is enough to choose $c=-a+1$ and use that $a^q+a=2$.

With this new description, we describe the fixed fields ${\rm Fix}(M_\ell)$ and ${\rm Fix}(S_\ell)$ of $M_\ell$ and $S_\ell$.

\begin{proposition}\label{prop:fixedfieldsH_q}
Denote by $\mathcal{H}_q:=\mathbb{F}_{q^{2}}(x,y)$ the Hermitian function field over $\mathbb{F}_{q^{2}}$ where $y^{q+1}-x^{q+1}+1=0$. Let $\rho:=x+y$ and $u:=y(x^{q^2}-x)/(x^{q+1}-1)=(x^{q^2}-x)/y^q$. Then 
\[\mathrm{Fix}(E_q)=\FF_{q^2}(\rho), \quad \mathrm{Fix}(S_\ell)=\FF_{q^2}(u), \quad \text{and} \quad \mathrm{Fix}(M_\ell)=\FF_{q^2}(u^{q+1}).\]
\end{proposition}
\begin{proof}
Clearly $E_q$ fixes $\rho$, whence $\FF_{q^2}(\rho) \subseteq \mathrm{Fix}(E_q)$. On the other hand, by direct computation, $(x/\rho)^q+(x/\rho)=(\rho^{q+1}+1)/\rho^{q+1}$.
Since $|E_q|=q$, this implies that $\FF_{q^2}(\rho) = \mathrm{Fix}(E_q)$. 
To prove the rest of the claim, we start by investigating the effect of applying an automorphism from $M_\ell$ on $u$. We have
\begin{equation}\label{eq:alpha}\alpha_{a,b,c,d}(u)=(cx+dy) \bigg( \frac{ax^{q^2}+by^{q^2}-ax-by}{(cx+dy)^{q+1}}\bigg)=\bigg( \frac{ax^{q^2}+by^{q^2}-ax-by}{c^qx^q+d^qy^q}\bigg).\end{equation}
Observe that \begin{align*}
(x^{q^2}-x)(c^qx^q+d^qy^q)&=c^qx^{q^2+q}+d^qx^{q^2}y^q-c^qx^{q+1}-d^qy^qx\\
&=c^q(y^{q+1}+1)^q+d^qx^{q^2}y^q-c^q(y^{q+1}+1)-d^qy^qx\\
&=c^qy^{q^2+q}+d^qx^{q^2}y^q-c^qy^{q+1}-d^qy^qx\\
&=y^q(c^qy^{q^2}+d^qx^{q^2}-c^qy-d^qx).
\end{align*}
Combining this with Equation \eqref{eq:alpha}, we see that
$$\frac{\alpha_{a,b,c,d}(u)}{u}=\frac{y^q(ax^{q^2}+by^{q^2}-ax-by)}{(x^{q^2}-x)(c^qx^q+d^qy^q)}=\frac{ax^{q^2}+by^{q^2}-ax-by}{c^qy^{q^2}+d^qx^{q^2}-c^qy-d^qx}=\zeta,$$
where in the last equality we used that $b=\zeta c^q$ and $a=a^{q^2}=d^q/\zeta^q=\zeta d^q.$ In particular Equation \eqref{eq:Sl} now implies that $\alpha_{a,b,c,d}(u)=u$ if and only if $\alpha_{a,b,c,d} \in S_\ell$. Hence $u \in \mathrm{Fix}(S_\ell)$ and $u^{q+1} \in \mathrm{Fix}(M_\ell)$. On the other hand, computing the divisor of $u$ in $\mathcal{H}_q$, we find that
\begin{equation}\label{eq:div_u}
   (u)_{\mathcal{H}_q}=  \left( y \frac{x^{q^2} - x}{x^{q+1} - 1} \right)_{\mathcal{H}_q} = \sum_{a^{q+1} = 1} P_{(a,0)} + \sum_{\substack{a \in \Fqq \\ a^{q+1} \neq 1}} \sum_{b^{q+1} = a^{q+1} - 1} P_{(a,b)} - (q^2-q) \cdot \sum_{j=1}^{q+1} P_{\infty}^j.
\end{equation}
Thus, the pole divisor of $u$ has degree $|S_\ell|$, whence $[\FF_{q^2}(x,y):\FF_{q^2}(u)]=|S_\ell|$ and $[\FF_{q^2}(x,y):\FF_{q^2}(u^{q+1})]=(q+1)|S_\ell|=|M_\ell|$. Then $\Fqq(u)=\mathrm{Fix}(S_\ell)$ and $\Fqq(u^{q+1})=\mathrm{Fix}(M_\ell)$.
\end{proof}

In later sections we will work with function fields with constant field $\FF_{q^{2n}}$, with $n\ge 3$ an odd integer. We will therefore on occasion work with the function field $\mathcal{H}_q^{(n)}:=\FF_{q^{2n}}\mathcal{H}_q=\FF_{q^{2n}}(x,y)$ obtained from $\mathcal{H}_q$ by extending its constant field from $\Fqq$ to $\FF_{q^{2n}}$.

\subsection{The GGS function field and its subfields $\mathcal{Y}_{n,s}$}\label{subsec:Yns}

Let $n \ge 3$ be an odd integer. The G\"uneri-Garcia-Stichtenoth (GGS) maximal function field $\mathcal{C}_n$ over $\mathbb{F}_{q^{2n}}$ was introduced in \cite{GGS} as a generalization of the Giulietti-Korchmaros (GK) maximal function field over $\mathbb{F}_{q^6}$. More precisely, writing $m=(q^n+1)/(q+1)$, one has $\mathcal{C}_n=\mathbb{F}_{q^{2n}}(x,y,z)$, where
\begin{equation*} 
z^{m}=y^{q^2}-y,\qquad y^{q+1}=x^q+x.
\end{equation*}
A family of Galois subfields $\mathcal{Y}_{n,s}$ of $\mathcal{C}_n$ was introduced and studied by Tafazolian, Teh\'eran-Herrera and Torres in \cite{TTT}. here we list a few facts on $\mathcal{Y}_{n,s}$ from \cite{TTT}, for future use.

Let $s$ be a divisor of $m=\frac{q^n+1}{q+1}$. The function field $\mathcal{Y}_{n,s}$ is defined as $\FF_{q^{2n}}(x,y,z)$ where
\begin{equation} \label{Yns}
 \left\{ \begin{array}{rcl}
z^{m/s}=y^{q^2}-y,\\
y^{q+1}=x^q+x.
\end{array}\right.
\end{equation}
It is an $\FF_{q^{2n}}$-maximal function field and it has genus 
\begin{equation}\label{eq:genusYns}
g(\mathcal{Y}_{n,s})=\dfrac{q^{n+2}-q^n - sq^3 + q^2+s-1}{2s}.
\end{equation}
The GGS function field is included in this family, since $\mathcal{Y}_{n,1}=\mathcal{C}_n$.
The following result concerning the automorphism groups of $\mathcal{Y}_{n,s}$ is from \cite{MTZ1}. 
\begin{theorem}\label{th:TMZ1} 
Let $q$ be a prime power, $n \geq 3$ odd, $m:=(q^n+1)/(q+1)$, and $s$ a divisor of $m$ with $s\ne m$.
If $3\nmid n$ or $\frac{m}{s}\nmid(q^2-q+1)$, then ${\rm Aut}(\mathcal{Y}_{n,s})$ has order $q^3(q^2-1)m/s$ and is isomorphic to $S_{q^3}\rtimes C_{(q^2-1)\frac{m}{s}}$.
    If $3\mid n$ and $\frac{m}{s}\mid (q^2-q+1)$, then ${\rm Aut}(\mathcal{Y}_{n,s})$ has order $(q^3+1)q^3(q^2-1)m/s$ and is isomorphic to ${\PGU}(3,q)\rtimes C_{m/s}$. 
\end{theorem}

\section{The function field $\tilde{\mathcal{Y}}_{n,s}$ and its automorphism group}\label{sec:Y}

Through the rest of the paper, let $q$ be a power of a prime $p$, $n \ge 3$ be an odd integer, and $m:=(q^n+1)/(q+1)$. The so-called BM function field $\tilde{C}_n:=\FF_{q^{2n}}(x,y,z)$, where
\begin{equation*}
z^{m}=y\bigg( \frac{x^{q^2}-x}{x^{q+1}-1} \bigg) \quad \text{and} \quad y^{q+1}=x^{q+1}-1,
\end{equation*}
is an $\mathbb{F}_{q^{2n}}$-maximal second generalization, given in \cite{BM}, of the GK maximal function field. 
In this section we construct the analogue $\tilde{\mathcal{Y}}_{n,s}$ of the function field $\mathcal{Y}_{n,s}$ above, as s Galois subfield of the BM function field $\tilde{C}_n$ instead of the GGS function field $\mathcal{C}_n$.
Computing its genus and automorphism group will allows us to provide new examples of $\mathbb{F}_{q^{2n}}$-maximal function fields that are not isomorphic to subfields of $\mathcal{H}_{q^n}$.
The non-isomorphism property will follow immediately from the value of the genus of $\tilde{\mathcal{Y}}_{n,s}$; while $\tilde{\mathcal{Y}}_{n,s}$ being a new example, will follow from the structure of its automorphism group, which is an invariant under isomorphism and hence will allow us to distinguish $\tilde{\mathcal{Y}}_{n,s}$ from the known examples with the same genus.

\subsection{The function field $\tilde{\mathcal{Y}}_{n,s}$ and its automorphism group.}

Let $s$ be a divisor of $m=(q^n+1)/(q+1)$.
The function field $\tilde{\mathcal{Y}}_{n,s}$ is defined as $\FF_{q^{2n}}(x,y,z)$ with
\begin{equation} \label{tildeYns}
\left\{ \begin{array}{rcl}
z^{m/s}=y\bigg( \frac{x^{q^2}-x}{x^{q+1}-1} \bigg),\\
y^{q+1}=x^{q+1}-1.
\end{array}\right.
\end{equation}
The function field $\tilde{\mathcal{Y}}_{n,s}$ is $\FF_{q^{2n}}$-maximal as a Galois subfield of the BM function field $\tilde{\mathcal{C}}_n$, and generalizes it in the sense that $\mathcal{\mathcal{Y}}_{n,1}= \tilde{\mathcal{C}}_n$. The Galois group of $\tilde{\mathcal{C}}_n/\tilde{\mathcal{Y}}_{n,s}$ is cyclic of order $s$ and is by the automorphism $(x,y,z) \mapsto (x,y,\lambda z)$ of $\tilde{\mathcal{C}}_n$, where $\lambda$ is a primitive $s$-th root of unity.
The genus of $\tilde{\mathcal{Y}}_{n,s}$ turns out to be the same genus as the genus of ${\mathcal{Y}}_{n,s}$.

\begin{proposition} \label{genusYtilde}
We have $g(\tilde{\mathcal{Y}}_{n,s})=g(\mathcal{Y}_{n,s})=\dfrac{q^{n+2}-q^n - sq^3 + q^2+s-1}{2s}.$
\end{proposition}

\begin{proof}
 Let $M:= m/s = (q^n+1)/(s(q+1))$. The Hermitian function field $\mathcal{H}_q = \Fqq(x,y)$ with $y^{q+1} = x^{q+1} - 1$ is a subfield of $\tilde{\mathcal{Y}}_{n,s}$. 
In the constant field extension $\mathcal{H}_q^{(n)} = \FF_{q^{2n}}\mathcal{H}_q$, the $q^3+1$ rational places of $\mathcal{H}_q$ extend uniquely to $q^3+1$ rational places of $\mathcal{H}_q^{(n)}$; we will use the same notation for them as in Subsection \ref{subsec:Hermitian}. This convention yields that Equation \eqref{eq:div_u} is also valid over $\mathcal{H}_q^{(n)}$. Equation \eqref{eq:div_u} and the defining equation of $\tilde{\mathcal{Y}}_{n,s}$, imply that the extension $\tilde{\mathcal{Y}}_{n,s}$ is a Kummer extension of $\mathcal{H}_q^{(n)}$ of degree $M$. As $q^2-q$ is relatively prime to $m$ we also have $\gcd(M, q^2-q) = 1$, showing that all places in the support of the divisor of $u$ are totally ramified under $\tilde{\mathcal{Y}}_{n,s}$. Using that $g(\mathcal{H}_q^{(n)}) = q(q-1)/2$, the genus of $\tilde{\mathcal{Y}}_{n,s}$ and the claim then follow directly from the Hurwitz genus formula and Equation \eqref{eq:genusYns}:
\begin{align*}
     g(\tilde{\mathcal{Y}}_{n,s}) &= 1 + M \left( \frac{q(q-1)}{2} - 1 \right) + \frac{1}{2}(M-1)(q^3 +1) \\
     &= \frac{q^{n+2} - q^n - sq^3 + q^2 + s - 1}{2s} = g(\mathcal{Y}_{n,s}).
\end{align*}
\end{proof}

From the proof of Proposition \ref{genusYtilde}, one can immediately deduce that every $\FF_{q^2}$-rational place $P$ of $\mathcal{H}_q$ is lying under a unique place of $\tilde{\mathcal{Y}}_{n,s}$. In case $P=P_\infty^j$, we denote such place by $\tilde{P}_\infty^j$. In case $P=P_{(a,b)}$ is an $\Fqq$-rational place of $\mathcal H_q$, we denote such place with $\tilde{P}_{(a,b,0)}$, since Equation \eqref{tildeYns} implies that the $z$-coordinate of that place is zero (i.e. $\tilde{P}_{(a,b,0)}$ is a zero of $z$). We define the two following sets of places of $\tilde{\mathcal{Y}}_{n,s}$:
\begin{equation}\label{eq:orbit_infzero}
O_\infty:=\{\tilde{P}_\infty^1,\ldots,\tilde{P}_\infty^{q+1}\},\qquad
O:=\{\tilde{P}_{(a,b,0)} \mid a,b \in \Fqq, \ b^{q+1}-a^{q+1}+1=0\}.
\end{equation}
Note that $|O_\infty|=q+1$ and $|O|=q^3-q$. 

\begin{theorem} \label{lifttildeY}
The group $\aut(\tilde{\mathcal{Y}}_{n,s})$ has a subgroup $G \cong \SL(2,q) \rtimes C_{(q^n+1)/s}$. The group $G$ acts on the places of $\tilde{\mathcal{Y}}_{n,s}$ with two short orbits, namely $O_\infty$ and $O$ defined in Equation \eqref{eq:orbit_infzero}.
\end{theorem}

\begin{proof}
Let $\alpha_{a,b,c,d} \in M_\ell$ be an automorphism of $\mathcal{H}_q$, where $b=\zeta c^q$ and $d=\zeta a^q$ for a $(q+1)$-st root of unity $\zeta$. As in the proof of Proposition \ref{prop:fixedfieldsH_q}, we have $\alpha_{a,b,c,d}(u)=\zeta u$.
Since in $\tilde{\mathcal{Y}}_{n,s}$ we have $z^{m/s}=u$ and $\alpha_{a,b,c,d}(u)=\zeta u$, we conclude that for any $\alpha_{a,b,c,d} \in M_\ell$ and any $\xi \in \mathbb{F}_{q^{2n}}$ such that $\xi^{m/s}=\zeta$, the map $\alpha_{a,b,c,d,\xi}: (x,y,z) \mapsto (ax+by,cy+dy,\xi z)$ is in $\aut(\tilde{\mathcal{Y}}_{n,s})$. Note that since $\zeta^{q+1}=1$, we have $\xi^{(q^n+1)/s}=1,$ implying that indeed $\xi \in \mathbb{F}_{q^{2n}}.$ Hence the groups
\begin{equation}\label{eq:groupGH}
G:=\{\alpha_{a,\zeta c^q,c,\zeta a^q,\xi} \mid \alpha_{a,\zeta c^q,c,\zeta a^q} \in M_\ell, \ \xi^{m/s}=\zeta \},\quad
H:=\{\alpha_{a,c^q,c,a^q,\xi} \mid \alpha_{a,c^q,c,a^q} \in S_\ell, \ \xi^{m/s}=1 \}
\end{equation}
are automorphism groups of $\tilde{\mathcal{Y}}_{n,s}$ with
$G \cong \SL(2,q) \rtimes C_{(q^n+1)/s}$ and $H \cong S_\ell \times C_{m/s} \cong \SL(2,q) \times C_{m/s}$. The claims on the short orbits of $G$ follows by direct computation using the matrix representation of the automorphisms.
\end{proof}

The subgroup $E_q$ of $S_\ell$ from Equation \eqref{eq:Eq} can be viewed as a subgroup of $S_\ell \rtimes C_{m/s}$ as well. Then $E_q$ fixes both $\rho=x+y$ and $z$ when acting on $\tilde{\mathcal{Y}}_{n,s}$. The group $E_q$ fixes one of the poles of $z$, and acts with long orbits on the other places. Relabeling if needed, we can assume that $E_q$ fixes $\tilde{P}_\infty^{1}$.

\begin{lemma} \label{contr1TILDE}
Denote with $E_q$ the elementary abelian $p$-subgroup of order $q$ in $\SL(2,q)$ fixing $\tilde{P}_{\infty}^1$. Let $\alpha \in E_q \setminus \{\mathrm{id}\}$. Then $i(\alpha)=(q^n+1)/s+1$.
 \end{lemma}

\begin{proof}
Note that $\alpha$ is of prime order $p$, since $E_q$ is elementary abelian. Note also, in view of Proposition \ref{prop:fixedfieldsH_q}, that ${\rm Fix}(E_q)=\mathbb{F}_{q^{2n}}(\rho,z)$, where $\rho=x+y$. Indeed, $\mathbb{F}_{q^{2n}}(x,y,z) / \mathbb{F}_{q^{2n}}(\rho,z)$ is an extension of degree $q=|E_q|$ and both $\rho$ and $z$ are fixed by $E_q$. By direct computation, we have $z^{m/s}=(\rho^{q^2-1}-1)/\rho^{q-1}$. Using the theory of Kummer extensions, this implies that $\mathbb{F}_{q^{2n}}(\rho,z)$ has genus $g(\mathbb{F}_{q^{2n}}(\rho,z))=(m/s-1)(q^2-1)/2$.

Let $\hat{E_q}$ be the automorphism group of $\mathbb{F}_{q^{2n}}(x,y)={\rm Fix}(C_{m/s})$ induced by $E_q$. Let $\hat\alpha,\hat\beta$ any two elements in $\hat E_q\setminus\{\mathrm{id}\}$. Then $\hat\alpha$ and $\hat\beta$ are contained in the same number of ramification groups 
of $\hat{E}_q$ at $P_\infty^1$; see \cite{MZ2016}. Therefore, any two elements $\alpha,\beta\in E_q\setminus\{\mathrm{id}\}$ are contained in the same number of ramification groups of $E_q$ at $\tilde{P}_\infty^1$. Recalling that $E_q$ acts with long orbits out of $\tilde{P}_\infty^1$, this implies that 
all elements $\alpha \in E_q \setminus \{\mathrm{id}\}$ give the same contribution $i(\alpha)=:A$ to the different divisor of the extension $\mathbb{F}_{q^{2n}}(x,y,z) / \mathbb{F}_{q^{2n}}(\rho,z)$. Since $g(\mathbb{F}_{q^{2n}}(\rho,z))=(m/s-1)(q^2-1)/2$ we get from the Hurwitz genus formula
\begin{align*}
2g(\tilde{\mathcal{Y}}_{n,s})-2&=\frac{(q^2-1)(q^n+1)}{s}-(q^3+1)=|E_q|\bigg( 2g(\mathbb{F}_{q^{2n}}(\rho,z))-2\bigg)+\sum_{\alpha \in E_q \setminus \{0\}}i(\alpha)\\
&=q\bigg( (q^2-1)(m/s-1)-2\bigg)+A(q-1),
\end{align*}
from which one gets the claim $A=(q^n+1)/s+1$.
\end{proof}

Our aim is to prove for $\tilde{\mathcal{Y}}_{n,s}$ the analogue \ref{mainYtilde} of Theorem \ref{th:TMZ1}.

\begin{theorem} \label{mainYtilde}
    If $3\nmid n$ or $\frac{m}{s}\nmid(q^2-q+1)$, then the full automorphism group ${\rm Aut}(\tilde{\mathcal{Y}}_{n,s})$ of $\tilde{\mathcal{Y}}_{n,s}$ has order $q(q^2-1)(q+1)m/s$ and is equal to $G \cong \SL(2,q) \rtimes C_{(q^n+1)/s}$.
    
    If $3\mid n$ and $\frac{m}{s}\mid (q^2-q+1)$, then ${\rm Aut}(\tilde{\mathcal{Y}}_{n,s})$ has order $(q^3+1)q^3(q^2-1)m/s$ and is an extension of a normal subgroup isomorphic to ${\PGU}(3,q)$ by a cyclic group of order $m/s$. 
\end{theorem}

We start with a small remark that will set our assumptions for the rest of the section.

\begin{remark} \label{aspetta}
In order to prove Theorem \ref{mainYtilde} we only have to analyze the case $3\nmid n$ or $\frac{m}{s}\nmid(q^2-q+1)$. Indeed if $3\mid n$ and $\frac{m}{s}\mid (q^2-q+1)$ then the function fields $\tilde{\mathcal{Y}}_{n,s}$ and ${\mathcal{Y}}_{n,s}$ are isomorphic. This follows from the fact that $\mathcal{C}_3$ and $\tilde{\mathcal{C}}_3$ are both isomorphic to the GK function field, and if $\frac{m}{s}\mid (q^2-q+1)$ then $\tilde{\mathcal{Y}}_{n,s}$ (resp. ${\mathcal{Y}}_{n,s}$) is a Galois subfield of $\tilde{\mathcal{C}}_3$ (resp. $\mathcal{C}_3$) seen as a maximal function field not only over $\mathbb{F}_{q^6}$, but over $\mathbb{F}_{q^{2n}}$.
\end{remark}

According to Remark \ref{aspetta}, we can assume that either $3\nmid n$ or $\frac{m}{s}\nmid(q^2-q+1)$, with the aim of proving that $G$ is the full automorphism group of $\tilde{\mathcal{Y}}_{n,s}$. Hence the condition $3\nmid n$ or $\frac{m}{s}\nmid(q^2-q+1)$ will be assumed from now on.

We denote by $C_{m/s}$ the Galois group of the extension $\mathbb{F}_{q^{2n}}(x,y,z) / \mathbb{F}_{q^{2n}}(x,y)$, which is generated by the automorphism $\tau:(x,y,z) \mapsto (x,y,\xi z)$, where $\xi$ is a primitive $(m/s)$-th root of unity. 
Now recall the notation $O_\infty$ and $O$ from Equation \eqref{eq:orbit_infzero}. Understanding the action of the automorphisms of $\tilde{\mathcal{Y}}_{n,s}$ on $O_\infty$ and $O$ is the key to prove Theorem \ref{mainYtilde}.

\begin{lemma} \label{normactionTILDE}
The group $C_{m/s}$ is normal in $\aut(\tilde{\mathcal{Y}}_{n,s})$ if and only if $\aut(\tilde{\mathcal{Y}}_{n,s})$ acts on the set $O \cup O_\infty$. Furthermore, if $C_{m/s}$ is normal in $\aut(\mathcal{Y}_{n,s})$ then the quotient group $\aut(\tilde{\mathcal{Y}}_{n,s})/C_{m/s}$ is an automorphism group of the Hermitian function field $\mathcal{H}_q^{(n)}=\mathbb{F}_{q^{2n}}(x,y)$ and either $\aut(\tilde{\mathcal{Y}}_{n,s})=G$ or $|\aut(\tilde{\mathcal{Y}}_{n,s})/C_{m/s}|=q^3(q^2-1)(q^3+1)$.
In the latter case $\aut(\tilde{\mathcal{Y}}_{n,s})/C_{m/s}$ is isomorphic to $\PGU(3,q)$ and acts as $\PGU(3,q)$ on the Hermitian function field.
\end{lemma}

\begin{proof}
Recall that $O \cup O_\infty$ consists exactly of the fixed places of $C_{m/s}$ on $\tilde{\mathcal{Y}}_{n,s}$, and $C_{m/s}$ has no other short orbits. 
This implies that, if $C_{m/s}$ is normal in $\aut(\tilde{\mathcal{Y}}_{n,s})$, then $\aut(\tilde{\mathcal{Y}}_{n,s})$ acts on $O \cup O_\infty$. Indeed if $\alpha \in \aut(\tilde{\mathcal{Y}}_{n,s})$, $\beta \in C_{m/s}$ and $\tilde{P} \in O \cup O_\infty$ then $\alpha^{-1} \circ \beta \circ \alpha=\beta^{\prime} \in C_{m/s}$ and $\tilde{P}=\beta^\prime(\tilde{P})=\alpha^{-1}(\beta(\alpha(\tilde{P})))$, which gives $\alpha(\tilde{P})=\beta(\alpha(\tilde{P}))$ and so $\alpha(\tilde{P}) \in O \cup O_\infty$.

Conversely, suppose that $\aut(\tilde{\mathcal{Y}}_{n,s})$ acts on the set $O \cup O_\infty$ and denote with $T$ the subgroup 
$$T=\{\sigma \in \aut(\tilde{\mathcal{Y}}_{n,s}) \mid \sigma(\tilde{P})=\tilde{P}, \ \textrm{for \ all} \ \tilde{P} \in O \cup O_\infty \}.$$
Then $T$ is a normal subgroup of $\aut(\tilde{\mathcal{Y}}_{n,s})$. Indeed let $g \in \aut(\tilde{\mathcal{Y}}_{n,s})$ and $\sigma \in T$. Then for all $\tilde{P} \in O \cup O_\infty$ it holds that $g(\tilde{P}) \in O \cup O_\infty$,
so that $\sigma(g(\tilde{P}))=g(\tilde{P})$. Hence the conjugate $g^{-1}\sigma g$ of $\sigma$ is such that $g^{-1}\sigma g(\tilde{P})=g^{-1}\sigma(g(\tilde{P}))=\tilde{P}$ for all $\tilde{P} \in O \cup O_\infty$, that is, $g \sigma g^{-1} \in T$. Note also that the order of $T$ is coprime with $p$ because $|O \cup O_\infty|>1$ and the function field has $p$-rank zero, see \cite[Lemma 11.129]{HKT}.  This implies that $T$ is a cyclic group by Lemma \ref{highnormal}. Also $C_{m/s} \subseteq T$ is a characteristic subgroup of $T$ and thus a normal subgroup of $\Aut(\tilde{\mathcal{Y}}_{n,s})$. 

To prove the second part of the claim, recall that the fixed field of $C_{m/s}$ is the Hermitian function field $\mathbb{F}_{q^{2n}}(x,y)$ and $\aut(\mathbb{F}_{q^{2n}}(x,y)) = \PGU(3,q)$, see \cite[Proposition 11.30]{HKT}.
Thus, $\aut(\tilde{\mathcal{Y}}_{n,s})/C_{m/s}$ acts faithfully on the places below $O\cup O_{\infty}$ as a subgroup of $\PGU(3,q)$ and contains $G/C_{m/s}$, which is a group of order $q(q-1)(q+1)^2$ acting on the $q+1$ places below $O_\infty$.
The subgroup of $\PGU(3,q)$ acting on such $q+1$ places is a maximal subgroup of order $q(q-1)(q+1)^2$, see \cite[Theorem A.10]{HKT}, and hence we get that either $\aut(\tilde{\mathcal{Y}}_{n,s})/C_{m/s}=G/C_{m/s}$ (and so $O_\infty$ and $O$ are two distinct orbits of $\aut(\tilde{\mathcal{Y}}_{n,s})$) or $\aut(\tilde{\mathcal{Y}}_{n,s})/C_{m/s}=\PGU(3,q)$.
\end{proof}

Next we make it more precise when the second case in Lemma \ref{normactionTILDE} occurs.

\begin{proposition} \label{fineTILDE}
Assume that the group $C_{m/s}$ is normal in $\aut(\tilde{\mathcal{Y}}_{n,s})$. Then
$\aut(\tilde{\mathcal{Y}}_{n,s})/C_{m/s}$ is isomorphic to $\PGU(3,q)$ if and only if $3 \mid n$ and $\frac{m}{s}\mid(q^2-q+1)$. 
\end{proposition}

\begin{proof}
If $3 \mid n$ and $\frac{m}{s} \mid q^2-q+1$, then Lemma \ref{normactionTILDE} implies that $\aut(\tilde{\mathcal{Y}}_{n,s})/C_{m/s}$ is isomorphic to $\PGU(3,q)$.
Conversely, suppose that $\aut(\tilde{\mathcal{Y}}_{n,s})/C_{m/s}$ is isomorphic to $\PGU(3,q)$; we have to shown that $3 \mid n$ and $\frac{m}{s}\mid(q^2-q+1)$.

The automorphism group $\PGU(3,q)$ of the Hermitian function field $\mathcal{H}_q$ is $2$-transitive on
the rational places of $\mathcal{H}_q$ over $\mathbb{F}_{q^2}$, and hence can be generated as $\PGU(3, q) = \langle \PGU(3,q)_{P_\infty^1}, \tau\rangle$,
where $\tau(x) = 1/x$ and $\tau(y) = \alpha y/x$ with $\alpha^{q+1}=-1$; see also \cite[Page 664, item 9]{HKT}. 
Since $C_{m/s}$ is normal in $\aut(\tilde{\mathcal{Y}}_{n,s})$, we apply Lemma \ref{normactionTILDE}. From $G/C_{m/s} = M_\ell$ we get that $\aut(\tilde{\mathcal{Y}}_{n,s})/C_{m/s}$ is isomorphic to $\PGU(3,q)$ if and only if $\tau$ can be lifted to an automorphism $\tilde{\tau}$ of $\tilde{\mathcal{Y}}_{n,s}$. Therefore, such $\tilde{\tau}$ exists. Its action on the places in $O \cup O_\infty$ is the same as the action of $\tau$ on the corresponding places of $\mathcal{H}_q^{(n)}$ below the places in $O \cup O_\infty$, since these are precisely the totally ramified places in the extension $\aut(\tilde{\mathcal{Y}}_{n,s})/\mathcal{H}_q^{(n)}$.
Note that
\begin{equation}\label{eq:div_z}
(z)_{\tilde{\mathcal{Y}}_{n,s}}=\sum_{\substack{a \in \Fqq}} \sum_{b^{q+1} = a^{q+1} - 1} \tilde{P}_{(a,b,0)} - (q^2-q) \cdot \sum_{j=1}^{q+1} \tilde{P}_{\infty}^j,
\end{equation}
where $\tilde{P}_{(a,b,0)}$ denotes the place above $P_{(a,b)}$ in $\tilde{\mathcal{Y}}_{n,s}/\mathcal{H}_q^{(n)}$.
Hence 
$$(\tilde{\tau}(z))_{\tilde{\mathcal{Y}}_{n,s}}=\sum_{\substack{a \in \Fqq \\ a \ne 0}} \sum_{b^{q+1} = a^{q+1} - 1} \tilde{P}_{(a,b,0)} +\sum_{j=1}^{q+1} \tilde{P}_{\infty}^j  - (q^2-q) \cdot \sum_{b^{q+1} = - 1} \tilde{P}_{(0, b,0)},$$
which implies that
$$\bigg(\frac{z}{\tilde{\tau}(z)}\bigg)_{\tilde{\mathcal{Y}}_{n,s}}=(q^2-q+1)\cdot \bigg(\sum_{b^{q+1} = - 1} \tilde{P}_{(0,b,0)}-\sum_{j=1}^{q+1} \tilde{P}_{\infty}^j \bigg).$$
On the other hand $$(x)_{\mathcal{H}_q}=\sum_{b^{q+1} = - 1} {P}_{(0,b)}-\sum_{j=1}^{q+1} {P}_{\infty}^j,$$
which implies that
$$(x)_{\tilde{\mathcal{Y}}_{n,s}}=\frac{m}{s}\bigg(\sum_{b^{q+1} = - 1} \tilde{P}_{(0,b,0)}-\sum_{j=1}^{q+1} \tilde{P}_{\infty}^j\bigg).$$
This implies that the order $d$ of the class of the divisor $$\sum_{b^{q+1} = - 1} \tilde{P}_{(0,b,0)}-\sum_{j=1}^{q+1} \tilde{P}_{\infty}^j$$ in the Picard group of $\tilde{\mathcal{Y}}_{n,s}$ divides $\gcd(m/s, q^2-q+1).$ It is easy to see that $\gcd(q^n+1,q^3+1)$ equals $q^3+1$ if $3 \mid n$ and $q+1$ otherwise. Hence $\gcd(m,q^2-q+1)$ equals $q^2-q+1$ if $3 \mid n$ and $1$ otherwise. This implies that the order $d$ divides $m/s$ and is strictly less than $m/s$, unless $3 \mid n$ and $m/s \mid q^2-q+1$. We can therefore finish the proof by showing that $d<m/s$ cannot occur. 
Assume by contradiction that $d<m/s$. Then there exists $f \in \tilde{\mathcal{Y}}_{n,s}$ such that $$(f)_{\tilde{\mathcal{Y}}_{n,s}}=d\bigg(\sum_{b^{q+1} = - 1} \tilde{P}_{(0,b,0)}-\sum_{j=1}^{q+1} \tilde{P}_{\infty}^j\bigg).$$
Denote with $\rho$ the function in $\mathcal{H}_q$ given by $\rho:=x+y$. Then \cite[Equation (3.1)]{BM} implies
\begin{equation}\label{eq:div_rho}
(\rho)_{\tilde{\mathcal{Y}}_{n,s}}=q\frac{m}{s}\tilde{P}_{\infty}^{1} - \frac{m}{s} \sum_{i=2}^{q+1} \tilde{P}_{\infty}^i.
\end{equation}
In particular
$$(f/\rho)_{\tilde{\mathcal{Y}}_{n,s}}=d\sum_{b^{q+1} = - 1} \tilde{P}_{(0,b,0)}-\bigg(q\frac{m}{s}+d\bigg)\tilde{P}_{\infty}^{1} + \bigg(\frac{m}{s}-d\bigg) \sum_{i=2}^{q+1} \tilde{P}_{\infty}^i.$$
We deduce that $qm/s+d \in H(\tilde{P}_{\infty}^{1})$. Since $\tilde{P}_{\infty}^{1}$ is totally ramified in $\tilde{\mathcal{C}}_n/\tilde{\mathcal{Y}}_{n,s}$ (recall that $\tilde{\mathcal{C}}_n$ is the BM function field), one has that $qm+sd$ is a non-gap at the place $\tilde{Q}_\infty^{(1)}$ of $\tilde{\mathcal{C}}_n$ above $\tilde{P}_{\infty}^{1}$. From \cite[Theorem 1.1]{MPL}, all positive nongaps of $\tilde{Q}_\infty^{(1)}$ strictly smaller than $q^n+1$ are of the form $mq+j(q^2-q)$ for a positive integer $j$. Since $d<m/s$, we have $sd < m$ and therefore $qm+sd<q^n+1$. Hence we conclude that $q^2-q$ divides $sd$.
This is a contradiction, since both $s$ and $d$  divide $m$ and $q$ does not divide $m$. 
The claim is proved.
\end{proof}

Thus, in order to prove Theorem \ref{mainYtilde} it is sufficient to show that $\aut(\tilde{\mathcal{Y}}_{n,s})$ acts on $O \cup O_\infty$. We start by showing that, unless $\aut(\tilde{\mathcal{Y}}_{n,s})=G$, the $\aut(\tilde{\mathcal{Y}}_{n,s})$-orbit containing $O_\infty$ is larger than $O_\infty$. We denote by $\aut(\tilde{\mathcal{Y}}_{n,s})_{O_\infty}$ the subgroup of $\aut(\tilde{\mathcal{Y}}_{n,s})$ stabilizing $O_\infty$ setwise.

\begin{lemma} \label{stabinfYtilda}
We have $\aut(\tilde{\mathcal{Y}}_{n,s})_{O_\infty}=G$. Thus, either $\aut(\tilde{\mathcal{Y}}_{n,s})=G$ or there exists a rational place $P$ of $\tilde{\mathcal{Y}}_{n,s}$ not in  $O_\infty$ such that $P$ is in the $\aut(\tilde{\mathcal{Y}}_{n,s})$-orbit containing $O_\infty$. 
\end{lemma}

\begin{proof}
Recall that $O \cup O_\infty$ contains exactly those places of $\tilde{\mathcal{Y}}_{n,s}$ fixed  by $C_{m/s}$. Let 
$$T_\infty:=\{\sigma \in \aut(\tilde{\mathcal{Y}}_{n,s})_{O_\infty} \mid \sigma(\tilde{P})=\tilde{P}, \ \textrm{for  all} \ \tilde{P} \in O_\infty \}\leq \aut(\tilde{\mathcal{Y}}_{n,s})_{O_\infty} .$$
Then, arguing as in the proof of Lemma \ref{normactionTILDE}, $T_\infty$ is normal in $\aut(\tilde{\mathcal{Y}}_{n,s})_{O_\infty}$. Also, the order of $T_\infty$ is coprime with the characteristic $p$ because $|O_\infty|>1$ and $\tilde{\mathcal{Y}}_{n,s}$ has $p$-rank zero, see \cite[Lemma 11.129]{HKT}.  Then $T_\infty$ is cyclic by Lemma \ref{highnormal}. Also, $C_{m/s}$ is a characteristic subgroup of $T_\infty$ and thus a normal subgroup of $\aut(\tilde{\mathcal{Y}}_{n,s})_{O_\infty}$.  
The fixed field of $C_{m/s}$ is $\mathcal{H}_q^{(n)}=\mathbb{F}_{q^{2n}}(x,y)$, and $\aut(\mathcal{H}_q^{(n)}) \cong \PGU(3,q)$; see \cite[Proposition 11.30]{HKT}. Hence $\aut(\tilde{\mathcal{Y}}_{n,s})_{O_\infty}/C_{m/s}$ contains $G/C_{m/s}$ and acts faithfully as a subgroup of $\PGU(3,q)$ on the set $\ell$ of the $q+1$ places of $\mathcal{H}_q^{(n)}$ below $O_\infty$.
As $\PGU(3,q)_{\ell}$ has order $q(q-1)(q+1)^2=|G/C_{m/s}|$ and is maximal in $\PGU(3,q)$ (see \cite[Theorem A.10]{HKT} and \cite[Theorem 3.1]{MZ}), we conclude that $\aut(\tilde{\mathcal{Y}}_{n,s})_{O_\infty}=G$.
\end{proof}

\begin{remark} \label{rem1}
Let $\Sigma$ be the $\aut(\tilde{\mathcal{Y}}_{n,s})$-orbit containing $O_\infty$; note that $\Sigma$ is a short orbit of $\aut(\tilde{\mathcal{Y}}_{n,s})$. By Lemma \ref{stabinfYtilda}, we can assume that $\Sigma$ contains a rational place (i.e., an $\FF_{q^{2n}}$-rational place) $P$ of $\tilde{\mathcal{Y}}_{n,s}$ with $P\notin O_\infty$. 
Also, if all such places $P$ lie in $O$, then $\Sigma=O \cup O_\infty$, meaning that $\aut(\tilde{\mathcal{Y}}_{n,s})$ acts on $O \cup O_\infty$, which is what we want to prove.

Therefore, in order to prove Theorem \ref{mainYtilde}, we assume from now on that $\Sigma$ contains a rational place out of $O_\infty \cup O$.
This implies that $\Sigma$ contains at least one long $G$-orbit, so that 
\begin{equation}
|\Sigma| \ge |O_\infty|+|G|=q+1+\frac{(q^n+1)(q^2-1)q}{s}.
\end{equation}
Together with $|G_{\tilde{P}_\infty^1}|=q(q-1)(q^n+1)/s$ and the orbit-stabilizer theorem, this yields
\begin{eqnarray*}
|\aut(\tilde{\mathcal{Y}}_{n,s})| & \ge & \frac{q(q-1)(q^n+1)}{s} \bigg(q+1+\frac{(q^n+1)(q^2-1)q}{s}\bigg)\\
& > & 2q g(\tilde{\mathcal{Y}}_{n,s})+\frac{2q^2(q-1)(q^n+1)}{s} \bigg(g(\tilde{\mathcal{Y}}_{n,s})+\frac{q^3+1}{2}\bigg)\\
& > & 2q g(\tilde{\mathcal{Y}}_{n,s})+\frac{2q^2(q-1)(q^n+1)}{s} g(\tilde{\mathcal{Y}}_{n,s})+\frac{2(q^3+1)q^2}{q+1}g(\tilde{\mathcal{Y}}_{n,s})\\
&>&84(g(\tilde{\mathcal{Y}}_{n,s})-1).
\end{eqnarray*}
Here we used the following: if $q>2$, then $m/s \ge 2$; if $q=2$, then $m/s\geq5$, because $m$ is odd and $m/s$ does not divide $3$ by Remark \ref{aspetta}.
The strategy will be to find a contradiction to such inequality.
\end{remark}

The next lemmas focus on the short orbits of  $\aut(\tilde{\mathcal{Y}}_{n,s})$.

\begin{lemma} \label{onenontameTilde}
The short orbit $\Sigma$ is the only non-tame orbit of $\aut(\tilde{\mathcal{Y}}_{n,s})$.
\end{lemma}
\begin{proof}
The short orbit $\Sigma$ of $\aut(\tilde{\mathcal{Y}}_{n,s})$ is non-tame, since the stabilizer of $\tilde{P}_\infty^1$ contains 
$E_q$ from Equation \eqref{eq:Eq} 
and therefore has order divisible by $p$.
Let $S_1$ be the Sylow $p$-subgroup of $\aut(\tilde{\mathcal{Y}}_{n,s})_{\tilde{P}_\infty^1}$, which is also a Sylow $p$-subgroup of $\aut(\tilde{\mathcal{Y}}_{n,s})$ by Corollary \ref{1actionp2}.
Suppose by contradiction that there exists another non-tame short orbit $\bar \Sigma \ne \Sigma$ of $\aut(\tilde{\mathcal{Y}}_{n,s})$. 
Let $\bar{P}\in\bar{\Sigma}$. As in the proof of Corollary \ref{1actionp2}, there is a Sylow $p$-subgroup $\bar{S}$ of $\aut(\tilde{\mathcal{Y}}_{n,s})$ stabilizing $\bar{P}$. The Sylow $p$-subgroups $S_1,\bar{S}$ are conjugate, say $\alpha^{-1}\bar{S}\alpha=S_1$ with $\alpha\in \aut(\tilde{\mathcal{Y}}_{n,s})$. Then $\alpha(\tilde{P}_\infty^1)=\bar{P}$, a contradiction to $\Sigma,\bar{\Sigma}$ being distinct orbits of $\aut(\tilde{\mathcal{Y}}_{n,s})$.
\end{proof}

\begin{lemma} \label{numborbitsTilde}
Suppose that $\Sigma$ contains a rational place not in $O_\infty \cup O$. Then $\aut(\tilde{\mathcal{Y}}_{n,s})$ has exactly two short orbits, namely the non-tame short orbit $\Sigma$ and a tame short orbit.
\end{lemma}

\begin{proof}
By Remark \ref{rem1} and Theorem \ref{Hurwitz} we conclude that $\aut(\tilde{\mathcal{Y}}_{n,s})$ has at most three short orbits, the fixed field of $\aut(\tilde{\mathcal{Y}}_{n,s})$ has genus zero and one of the following cases occurs:
\begin{enumerate}
\item $p\geq3$, exactly three short orbits, two tame of length $|\aut(\tilde{\mathcal{Y}}_{n,s})|/2$ and one non-tame;
\item  exactly two short orbits, both non-tame; 
\item  only one short orbit which is non-tame, whose length divides $2g-2$; 
\item  exactly two short orbits, one tame and one non-tame.
\end{enumerate}

We start by observing that Case (2) cannot occur by Lemma \ref{onenontameTilde}. If Case (3) occurs then $|\Sigma|$ divides $2g-2$, which is not possible since $|\Sigma| \geq |O_\infty|+|G| =1+q+|G|>2g-2$.
Hence to complete the proof we have to rule out Case (1). 

Suppose by contradiction that Case (1) occurs; in particular, $q>2$. Since $q$ divides the order of $\aut(\tilde{\mathcal{Y}}_{n,s})_{\tilde{P}_\infty^1}$, the Hurwitz genus formula applied to $\tilde{\mathcal{Y}}_{n,s}/{\rm Fix}(\aut(\tilde{\mathcal{Y}}_{n,s}))$ 
implies that
\begin{eqnarray*}
2g(\tilde{\mathcal{Y}}_{n,s})-2 & \ge & -2|\aut(\tilde{\mathcal{Y}}_{n,s})|+ |\Sigma|(|\aut(\tilde{\mathcal{Y}}_{n,s})_{\tilde{P}_\infty^1}|-1+q-1)+2\frac{|\aut(\tilde{\mathcal{Y}}_{n,s})|}{2}(2-1)\\
& = & |\Sigma|(q-2) \ \ge \ (q+1+|G|)(q-2),
\end{eqnarray*}
using the orbit-stabilizer theorem for the first inequality. As $q>2$, this yields a contradiction.
\end{proof}

We are now in the position to complete the proof of Theorem \ref{mainYtilde}. 

\begin{proposition} \label{normalizzanoTilde}
The group $\aut(\tilde{\mathcal{Y}}_{n,s})$ acts on $O_\infty \cup O$.
\end{proposition}

\begin{proof}
We need to find a contradiction with the following facts, that we can assume because of the sequence of lemmas and propositions proven so far: $3\nmid n$ or $m/s \nmid (q^2-q+1)$; and $\aut(\tilde{\mathcal{Y}}_{n,s})$ has exactly two short orbits: one non-tame short orbit $\Sigma$ containing $O_\infty$ and of cardinality $|\Sigma|=q+1+i|O|+k|G|$ for some $i\in\{0,1\}$ and some $k \geq 1$, and one tame short orbit $O_1$ of cardinality $|O_1|=(1-i)|O|+\ell|G|$, for some $\ell \geq 0$.
We distinguish two cases.

\bigskip
\noindent 
\textbf{Case 1: $i=0$.} In this case $O_\infty\subsetneq\Sigma$ and $O\subseteq O_1$. 

Then, for any $Q \in O \subseteq O_1$, we get from the orbit-stabilizer theorem that
$$(q+1+k|G|)|\aut(\tilde{\mathcal{Y}}_{n,s})_{\tilde{P}_{\infty}^1}|=|\aut(\tilde{\mathcal{Y}}_{n,s})|=(|O|+\ell|G|)|\aut(\tilde{\mathcal{Y}}_{n,s})_Q|,$$
that is
$$\bigg(1+k\frac{q(q-1)(q^n+1)}{s}\bigg)|\aut(\tilde{\mathcal{Y}}_{n,s})_{\tilde{P}_{\infty}^1}|=\bigg(q^2-q+\ell \frac{q(q-1)(q^n+1)}{s} \bigg)|\aut(\tilde{\mathcal{Y}}_{n,s})_Q|.$$
Note that $G_{\tilde{P}_{\infty}^1}\leq \aut(\tilde{\mathcal{Y}}_{n,s})_{\tilde{P}_{\infty}^1}$, where $|G_{\tilde{P}_{\infty}^1}|=q(q-1)(q^n+1)/s$; let $r\geq1$ be the index of $G_{\tilde{P}_{\infty}^1}$ in $ \aut(\tilde{\mathcal{Y}}_{n,s})_{\tilde{P}_{\infty}^1}$.
Also, note that 
$G_Q\leq\aut(\tilde{\mathcal{Y}}_{n,s})_Q$, where $G_Q\cong C_{(q^n+1)/s}$ is generated by $\alpha_{1,0,0,\xi^{m/s},\xi}$ for some primitive $(q^n+1)/s$-th root of unity $\xi$; let $h\geq1$ be the index of $G_Q$ in $\aut(\tilde{\mathcal{Y}}_{n,s})_Q$.
Then
%
\begin{equation} \label{ost}
\bigg(1+k\frac{q(q-1)(q^n+1)}{s}\bigg)r=\bigg(1+\ell \frac{(q^n+1)}{s} \bigg)h.
\end{equation}
%
Recall also that $\aut(\tilde{\mathcal{Y}}_{n,s})_Q$ is cyclic, since $O_1$ is a tame orbit.

If $\ell=0$, then Equation \eqref{ost} yields $h=\bigg(1+k\frac{q(q-1)(q^n+1)}{s}\bigg)r > \frac{q(q-1)(q^n+1)}{s}$; using $m/s >1$, this implies 
$|\aut(\tilde{\mathcal{Y}}_{n,s})_Q|>q(q-1)(q^n+1)^2/s^2>4g+4$. This is a contradiction by \cite[Theorem 11.79]{HKT}.

Thus, $\ell \geq 1$. Since $\aut(\tilde{\mathcal{Y}}_{n,s})_Q$ is cyclic and contains $\alpha_{1,0,0,\xi^{m/s},\xi}$, the subgroup $C_{m/s}=\langle \alpha_{1,0,0,\xi^{m/s},\xi}^{q+1} \rangle$ is normal in $\aut(\tilde{\mathcal{Y}}_{n,s})_Q$. This implies that $\aut(\tilde{\mathcal{Y}}_{n,s})_Q/C_{m/s}$ is an automorphism group of the Hermitian function field $\mathbb{F}_{q^2}(x,y)$, it fixes a rational place, and has order $h(q+1)$. By \cite[Theorem A.10]{HKT}, $h$ divides $q-1$. 

Suppose that $h > 2$ and let $b$ be a prime divisor of $h$ (and hence of $q-1$). Let $\alpha$ be an element in $\aut(\tilde{\mathcal{Y}}_{n,s})_Q$ of order $b$. 
From \cite{MZ2016} the automorphism $\alpha_1$ induced by $\alpha$ on $\mathbb{F}_{q^2}(x,y)$ has exactly two fixed places and they are both $\mathbb{F}_{q^2}$-rational, that is, corresponding to two places $Q,R\in O_\infty \cup O$ that are also fixed by $\alpha$. Furthermore, $\langle\alpha_1\rangle$ has no other short orbits on the Hermitian function field, so that $\langle\alpha\rangle$  has no other short orbits on $\tilde{\mathcal{Y}}_{n,s}$.
In particular, $O_1\setminus\{Q,R\}$ is divisible by $b$. This is a contradiction, as $b$ divides $|O_1|$ and $|O_1\setminus\{Q,R\}|$ is equal to $|O_1|-1$ or $|O_1|-2$.
Then, either $h=1$ or $h=2$ (and $p$ is odd in the latter case). 

From Equation \eqref{ost} taken modulo $(q^n+1)/s$, we have $r=t(q^n+1)/s+h$ with $t \geq 0$. 
If $t=0$, i.e. $r=h$, then $r$ is coprime with $p$, so that $E_q$ is a normal subgroup of $\aut(\tilde{\mathcal{Y}}_{n,s})_{\tilde{P}_{\infty}^1}$ being its Sylow $p$-subgroup, see Lemma \ref{highnormal}. Hence $\aut(\tilde{\mathcal{Y}}_{n,s})_{\tilde{P}_{\infty}^1}/E_q$ is an automorphism group of the function field $\mathbb{F}_{q^{2n}}(\rho,z)$ with genus $(q^2-1)(m/s-1)/2$. From \cite[Theorem 11.79]{HKT}, $|\aut(\tilde{\mathcal{Y}}_{n,s})_{\tilde{P}_{\infty}^1}/E_q|=r(q-1)(q^n+1)/s<4g(\mathbb{F}_{q^{2n}}(\rho,z))+4<2(q^n+1)(q-1)/s$. Hence $r=h=1$ and the Hurwitz genus formula together with Lemma \ref{contr1TILDE} gives
\begin{multline*}
\frac{(q^n+1)(q^2-1)}{s}-(q^3+1)=-2|\aut(\tilde{\mathcal{Y}}_{n,s})|+\bigg( q+1+k\frac{q(q^2-1)(q^n+1)}{s}\bigg)\bigg(|\aut(\tilde{\mathcal{Y}}_{n,s})_{\tilde{P}_{\infty}^1}| \\
-1+(q-1)\frac{q^n+1}{s}\bigg)+\bigg(q^3-q+\ell \frac{q(q^2-1)(q^n+1)}{s}\bigg)\bigg(|\aut(\tilde{\mathcal{Y}}_{n,s})_Q|-1\bigg),\end{multline*}
and hence
$$k\bigg((q-1)\frac{q^n+1}{s}-1\bigg)-\ell=0.$$
Combining this with Equation \eqref{ost} and $r=h=1$, we obtain $k=\ell=0$, a contradiction.

Then we can assume $r=t(q^n+1)/s+h$ with $t \geq 1$. Write $r=p^e r_1$ where $e \geq 0$, $\gcd(r_1,p)=1$. 
Since $\aut(\tilde{\mathcal{Y}}_{n,s})_{\tilde{P}_{\infty}^1}$ contains a cyclic subgroup of order $r_1(q-1)(q^n+1)/s$ we get $r_1<2(q+1)$ from \cite[Theorem 11.79]{HKT}. 
Then 
$$(q^n+1)/s+h \leq r=p^e r_1 < 2p^e(q+1),$$ 
implying that $p^e \geq \lceil((q^n+1)/s+h)/(2(q+1)) \rceil$.

Let $S$ be the Sylow $p$-subgroup of $\aut(\tilde{\mathcal{Y}}_{n,s})_{\tilde{P}_{\infty}^1}$, of order $qp^e$. 
Then $\aut(\tilde{\mathcal{Y}}_{n,s})_{\tilde{P}_{\infty}^1}/S$ is a cyclic automorphism group of order $(q-1)(q^n+1)r_1/s$ of the fixed field ${\rm Fix}(S)$ of $S$. It fixes the place of ${\rm Fix}(S)$ lying below $\tilde{P}_{\infty}^1$. If $\tilde g$ is the genus of ${\rm Fix}(S)$, \cite[Theorem 11.79]{HKT} implies
$$(q-1)(q^n+1)r_1/s \leq 4\tilde g+4.$$
The Hurwitz genus formula applied with respect to $S$, together with Lemma \ref{contr1TILDE}, gives
\begin{eqnarray}\label{inequalities}
\frac{(q^n+1)(q^2-1)}{s} & > & \frac{(q^n+1)(q^2-1)}{s}-(q^3+1) \notag\\
& \geq & qp^e(2\tilde g-2)+2qp^e-2+(q-1)\bigg(\frac{q^n+1}{s}-1\bigg)\notag\\
& \geq & \frac{qp^e((q-1)(q^n+1)r_1-4s)}{2s}-2+(q-1)\bigg(\frac{q^n+1}{s}-1\bigg)\\
& \geq & \frac{q^n+1}{s} \bigg( \frac{qp^e((q-1)r_1-\frac{4}{q+1})}{2}+q-2\bigg).\notag
\end{eqnarray} 
Note that $\lceil((q^n+1)/s+h)/(2(q+1)) \rceil >1$ so that $p^e \geq p$.
If $p\geq3$ or $e>1$, then one gets
$$\frac{(q^n+1)(q^2-1)}{s}>\frac{q^n+1}{s} \bigg( \frac{3q(q-2)}{2}+q-2\bigg),$$
which is not possible unless $q=p^e=3$.
So necessarily either $q=p^e=3$, or $p=p^e=2$.
In both cases, Equation \eqref{inequalities} yields
$$
q^2-1>\frac{qp^e((q-1)r_1-\frac{4}{q+1})}{2}+q-2.
$$
Since $\gcd(r_1,p)=1$, this implies $r_1=1$ in both cases $q=p^e=3$ and $p=p^e=2$.
Then
$$
q\geq p^e=r=t(q^n+1)/s+h\geq (q^n+1)/s+1\geq (q+1)+1,
$$
a contradiction.

\bigskip
\noindent
\textbf{Case 2: $i=1$.} This means that $O$ and $O_\infty$ are in the same orbit under $\aut(\tilde{\mathcal{Y}}_{n,s})$. This implies that $H(\tilde{P}_\infty^1)=H(\tilde{P})$ for all $\tilde{P} \in O$. We want to show that this is not possible.

We start by constructing an element of $H(\tilde{P}_\infty^1)$. Define 
$$A:=\left\lceil \frac{s(q^2-q)}{m} \right\rceil,\qquad B:=s-A \cdot \frac{m}{q^2-q}.$$
Note that $A \ge 1$ and $-m/(q^2-q)+1 \le B \le 0$.
Combining the divisors of the functions $z$ and $\rho$ on $\tilde{\mathcal{Y}}_{n,s}$ as given in Equations \eqref{eq:div_z} and \eqref{eq:div_rho}, we obtain the divisor 
$$\bigg(\frac{z}{\rho^A}\bigg)_{\tilde{\mathcal{Y}}_{n,s}}
=
\sum_{\substack{a \in \Fqq}} \sum_{b^{q+1} = a^{q+1} - 1} \tilde{P}_{(a,b,0)}
+
\left(A\frac{m}{s}-(q^2-q)\right)\sum_{j=2}^{q+1} \tilde{P}_{\infty}^j
-
\left(Aq\frac{m}{s}+q^2-q\right)\tilde{P}_{\infty}^1.$$
Since
$$A\frac{m}{s}-(q^2-q) \ge \frac{s(q^2-q)}{m}\frac{m}{s}-(q^2-q) \ge 0,$$
we obtain that $\tilde{P}_{\infty}^1$ is the only pole of $z/\rho^A$, so that $Aq m/s+q^2-q \in H(\tilde{P}_\infty^1)$. We claim that $Aq m/s+q^2-q \not\in H(\tilde{P})$ for $\tilde{P} \in O$, which will conclude the proof.
To this aim, it is sufficient by Lemma \ref{gapsdiff} to show that there exists a regular differential $\omega$ such that $v_{\tilde{P}}(\omega)=Aqm/s+q^2-q-1$.

Using Proposition \ref{prop:fixedfieldsH_q} it is easily seen that $\Fix(G)=\FF_{q^{2n}}(z^{(q^n+1)/s})$, so that, by Theorem \ref{lifttildeY}, the only places of $\mathbb{F}_{q^{2n}}(z)$ which are possibly ramified under $\tilde{\mathcal{Y}}_{n,s}$ lie below $O_\infty \cup O$, i.e. they are the zero and the pole of $z$. However, the zero of $z$ is unramified in the extension $\tilde{\mathcal{Y}}_{n,s}/\mathbb{F}_{q^{2n}}(z)$ of degree $|{\rm SL}(2,q)|=q^3-q$, by Equation \eqref{eq:div_z}.
Hence the unique ramified place in $\tilde{\mathcal{Y}}_{n,s}/\mathbb{F}_{q^{2n}}(z)$ is the pole of $z$, lying below the places in $O_\infty$. 
Moreover, $G$ has a subgroup of order $q+1$ acting sharply transitively on $O_\infty$, while fixing $z$. 
By Equation \eqref{eq:div_z}, this implies that the support of $(dz)$ is $O_\infty$. As $\deg(dz)=2g(\tilde{\mathcal{Y}}_{n,s})-2$ and $|O_\infty|=q+1$, we obtain that 
$$(dz)_{\tilde{\mathcal{Y}}_{n,s}}=\frac{2g(\tilde{\mathcal{Y}}_{n,s})-2}{q+1}(\tilde{P}_\infty^1+\ldots+\tilde{P}_\infty^{q+1}),$$
so that the differential $fdz$ is regular if and only if $f$ is in $L(\frac{2g(\tilde{\mathcal{Y}}_{n,s})-2}{q+1}(\tilde{P}_\infty^1+\ldots+\tilde{P}_\infty^{q+1}))$. 
Thus, we want to construct a function $f \in L(\frac{2g(\tilde{\mathcal{Y}}_{n,s})-2}{q+1}(\tilde{P}_\infty^1+\ldots+\tilde{P}_\infty^{q+1}))$ such that $v_{\tilde{P}}(f)=Aqm/s+q^2-q-1$.
Let $P:=P_{(a,b)}$ be the restriction of $\tilde{P}$ to the Hermitian function field $\mathcal{H}_q$.
Since $2g(\mathcal{H}_q)<q^2$, $q^2$ is a non-gap at $P$; 
let $f_0$ be a function on $\mathcal{H}_q$ with 
$$(f_0)_{\mathcal{H}_q}=-q^2P+E_0,\quad \textrm{where}\quad E_0\geq0,\quad P\notin{\rm supp}(E_0).$$
Defining the function $t_0:=a^q(x-a)+b^q(y-b)$, one has
$$(t_0)_{\mathcal{H}_q}=(q+1)P-\sum_{i=1}^{q+1} P_\infty^i,$$
and hence
\begin{equation}\label{eq:divt0}
(t_0)_{\tilde{\mathcal{Y}}_{n,s}}=(q+1)\frac{m}{s}\tilde{P}-\frac{m}{s}\sum_{i=1}^{q+1}\tilde{P}_\infty^i,\qquad
(f_0 \cdot t_0^q)_{\tilde{\mathcal{Y}}_{n,s}}=q\frac{m}{s}\tilde{P}+\tilde{E}_0-q\frac{m}{s}\sum_{i=1}^{q+1}\tilde{P}_\infty^i,
\end{equation}
where $\tilde{E}_0$ is an effective divisor obtained by lifting the divisor $E_0$ on $\tilde{\mathcal{Y}}_{n,s}$.

Recall that $3 \nmid n$ or $m/s \nmid q^2-q+1$. Since $s$ is odd, we can also assume that $2\leq m/s\leq m/3$, avoiding the trivial cases $s=1$ and $s=m$.
The number $A=\lceil s(q^2-q)/m\rceil$ satisfies $A=1$ or $A>1$ according respectively to $m/s \ge q^2-q$ or $m/s<q^2-q$. 
We construct the function
$$f_A:=t_0^{A-1}(f_0t_0^q)z^{-(A-1)m/s+q^2-q-1}.$$
Note that $-(A-1)m/s+q^2-q-1 > -(q^2-q)+q^2-q-1=-1$, so that the exponent $-(A-1)m/s+q^2-q-1$ is not negative.
Then, using Equations \eqref{eq:div_z} and \eqref{eq:divt0}, 
one gets
\begin{multline*}
(f_A)_{\tilde{\mathcal{Y}}_{n,s}} = 
\bigg(Aq\frac{m}{s}+q^2-q-1\bigg)\tilde{P} +\tilde{E}_A \\ -\bigg((-(A-1)q^2+Aq+A-1)\frac{m}{s}+(q^2-q-1)(q^2-q) \bigg) \sum_{i=1}^{q+1} \tilde{P}_\infty^i,
\end{multline*}
for some effective divisor $\tilde{E}_A$ whose support does not contain $\tilde{P}$. What we need to show is
$$(-(A-1)q^2+Aq+A-1)\frac{m}{s}+(q^2-q-1)(q^2-q) \le \frac{2g(\tilde{\mathcal{Y}}_{n,s})-2}{q+1}= (q^2-1)\frac{m}{s}-(q^2-q+1),$$
that is equivalent to
\begin{equation}\label{eq:poleorder_goal}
A\frac{m}{s} \ge q^2-q+1+\frac{2}{q^2-q-1}.
\end{equation}
By definition of $A$, the integer $A\frac{m}{s}$ is at least $q^2-q$.
If $A\frac{m}{s}=q^2-q$, then from $\gcd(m,q^2-q)=1$ we get $m/s=1$, contradicting the assumption $m/s\geq2$.
If $A\frac{m}{s}=q^2-q+1$, then $m/s$ divides $q^2-q+1$ and hence our assumptions imply $3\nmid n$; but $3\nmid n$ yields $\gcd(m,q^2-q+1)=1$, whence  $m/s=1$, a contradiction.
Then $A\frac{m}{s}\geq q^2-q+2$, and the inequality \eqref{eq:poleorder_goal} holds.
\end{proof}

This completes the proof of Theorem \ref{mainYtilde}. 
We highlight that the function fields $\tilde{\mathcal{Y}}_{n,s}$ are new examples of $\mathbb{F}_{q^{2n}}$-maximal function fields.

\begin{theorem}
Suppose that $3 \nmid n$ or $m/s \nmid q^2-q+1$. Then $\tilde{\mathcal{Y}}_{n,s}$ and $\mathcal{Y}_{n,s}$ are not isomorphic.
\end{theorem}

\begin{proof}
This follows from Theorems \ref{th:TMZ1} and \ref{mainYtilde}.
\end{proof}

\begin{remark}
We already observed that $\tilde{\mathcal{Y}}_{n,s}$ and $\mathcal{Y}_{n,s}$ are isomorphic if $3$ divides $n$ and $m/s$ divides $q^2-q+1$. In particular, $\tilde{\mathcal{Y}}_{3,s}$ and $\mathcal{Y}_{3,s}$ are isomorphic. Therefore, by \cite[Theorem 4.4]{TTT}, we conclude that $\tilde{\mathcal{Y}}_{3,s}$ is not isomorphic to a subfield of the Hermitian function field $\cH_{q^3}$ whenever the divisor $s$ of $q^2 - q + 1$ satisfies $q > s(s+1)$.
\end{remark}

\section{The function field $\tilde{\mathcal{X}}_{a,b,n,s}$}\label{sec:X}

Let $q = p^a$ for a prime number $p$ and $a \geq 1$. Let $b$ be a divisor of $a$ and $f = a/b$. As in the previous section we let $n \geq 3$ be odd and $s$ be a divisor of $m = (q^n + 1)/(q + 1)$.

Our aim is to determine an analogue $\tilde{\mathcal{X}}_{a,b,n,s}$ of $\mathcal{X}_{a,b,n,s}$ from \cite{TTT} in the case of the BM function field. We will
construct $\tilde{\mathcal{X}}_{a,b,n,s}$ as the fixed field of a certain $p$-subgroup of $\aut(\tilde{\mathcal{Y}}_{n,s})$, just as $\mathcal{X}_{a,b,n,s}$
has been constructed in \cite{TTT} as the fixed field of a $p$-subgroup of $\aut({\mathcal{Y}}_{n,s})$.

In the proof of Theorem \ref{lifttildeY}, we determined a subgroup \( S_{\ell} \subseteq \text{Aut}(\tilde{\mathcal{Y}}_{n,s}) \) with \( S_{\ell} \cong \text{SL}(2, q) \) and \( |S_{\ell}| = q(q^2 - 1) \):
\begin{equation}\label{eq:Sl}
S_{\ell} = \{ (x, y, z) \mapsto (\lambda x + \mu^q y,\mu x + \lambda^q y, z) \mid \lambda,\mu \in \mathbb{F}_{q^2}, \lambda^{q+1} - \mu^{q+1} = 1 \}.
\end{equation}


Let $\tilde{E}_q$ be a Sylow $p$-subgroup of $S_\ell$, which is elementary abelian of order $q$. 
Up to conjugation in $S_\ell$, we can assume
\begin{equation}\label{eq:sylowX}
\tilde{E}_q = \{ \sigma_c:(x, y, z) \mapsto ((c+1)x + c y, -cx + (1-c)y, z) \mid c \in \mathbb{F}_{q^2}, c^q + c = 0 \},
\end{equation}
as the right-hand side of \eqref{eq:sylowX} is a subgroup of $S_\ell$ of order $q$.
Given $c_0\in\mathbb{F}_{q^2}$ with $c_0^{q-1}=-1$, we can construct a subgroup $\tilde{E}_{p^b}\leq\tilde{E}_q$ of order $p^b$ as
\begin{equation}\label{eq:Epb}
\tilde{E}_{p^b} := \{ (x, y, z) \mapsto \left( \left( (\beta/c_0) + 1 \right) x + (\beta/c_0) y, -(\beta/c_0) x + \left(1 - (\beta/c_0) \right)y, z\right) \mid \beta \in \mathbb{F}_{p^b} \}.
\end{equation}
Define $\rho:=x+y$, and note that $\sigma_c\in \tilde{E}_q$ acts on $\tilde{\mathcal{Y}}_{n,s}$ as
$$ \sigma_c:(x,y,z)\mapsto (x+c\rho,y-c\rho,z).$$
Then $\mathbb{F}_{q^{2n}}(\rho,z)$ is a subfield of ${\rm Fix}(\tilde{E}_q)$. In fact $\mathbb{F}_{q^{2n}}(\rho,z)={\rm Fix}(\tilde{E}_q)$, as
\begin{equation}\label{eq:rho}(x/\rho)^q+x/\rho=(\rho^{q+1}+1)/\rho^{q+1}\end{equation}
 and hence $[\tilde{\mathcal{Y}}_{n,s}:\mathbb{F}_{q^{2n}}(\rho,z)]=q$.
In the next lemma we determine the ${\rm Fix}(\tilde{E}_{p^b})$. 
To this aim, fix an element $c_0 \in \mathbb{F}_{q^2}$ with $c_0^{q-1}=-1$, and define $$h_{c_0}:=c_0x/\rho-(c_0x/\rho)^{p^b}.$$
\begin{lemma} \label{lemma41}
The fixed field of $\tilde{E}_{p^b}$ is $\mathbb{F}_{q^{2n}}(\rho,z,h_d)$.
\end{lemma}

\begin{proof}
By direct computation, $\mathbb{F}_{q^{2n}}(\rho,z,h_d)$ is fixed by $\tilde{E}_{p^b}$.
Since $[\tilde{\mathcal{Y}}_{n,s}:\mathbb{F}_{q^{2n}}(\rho,z)]=q$, it is enough to prove that $[\mathbb{F}_{q^{2n}}(\rho,z,h_d):\mathbb{F}_{q^{2n}}(\rho,z)]=q/p^b=p^{(f-1)b}$.
To this aim, we show that
$$\ell(T)=T^{p^{(f-1)b}}+T^{p^{(f-2)b}}+\ldots+T^{p^b}+T - c_0(\rho^{q+1}+1)/\rho^{q+1}\in \mathbb{F}_{q^{2n}}(\rho,z)[T]  $$
is the minimal polynomial of $h_{c_0}$.
By Equation \eqref{eq:rho}, $h_{c_0}$ is a root of $\ell(T)$.
Since $\mathbb{F}_{q^{2n}}(\rho,z)$ is defined by 
$z^{m/s}=(\rho^{q^2-1}-1)/\rho^{q-1}$, the function $c_0(\rho^{q+1}+1)/\rho^{q+1}$ has multiplicity coprime with $p$ at its pole in $\mathbb{F}_{q^{2n}}(\rho,z)$,
so that $\ell(T)$ is irreducible by \cite[Proposition 3.7.10]{stichtenoth}.
\end{proof}
At this point we can define 
the curve $\tilde{\mathcal{X}}_{a,b,n,s}$ as
$$\tilde{\mathcal{X}}_{a,b,n,s}:=\mathbb{F}_{q^{2n}}\bigg( \rho,c_0\frac{x}{\rho}-\bigg(c_0\frac{x}{\rho} \bigg)^{p^b},z\bigg) \subseteq \tilde{\mathcal{Y}}_{n,s},$$
that is the function field $\tilde{\mathcal{X}}_{a,b,n,s}:=\mathbb{F}_{q^{2n}}(u,v,z)$ defined by
\begin{equation} \label{eqXtilda}
\begin{cases} z^{m/s}=\frac{u^{q^2-1}-1}{u^{q-1}}, \\ \sum_{i=0}^{f-1} v^{p^{ib}}=c_0\frac{u^{q+1}+1}{u^{q+1}},\end{cases}
\end{equation}
where $u=\rho$ and $v=h_{c_0}=c_0x/\rho-(c_0 x/\rho)^{p^b}$.
As a subfield of $\tilde{\mathcal{Y}}_{n,s}$, $\tilde{\mathcal{X}}_{a,b,n,s}$ is $\mathbb{F}_{q^{2n}}$-maximal.
Next lemma shows that $\tilde{\mathcal{Y}}_{n,s}$ and $\mathcal{Y}_{n,s}$ have the same genus.

\begin{lemma} \label{genXtilda}
 The genus of $\tilde{\mathcal{X}}_{a,b,n,s}$ is
 $$g(\tilde{\mathcal{X}}_{a,b,n,s})=\frac{q^{n+2}-p^bq^n-sq^3+q^2+(s-1)p^b}{2sp^b}=g({\mathcal{X}}_{a,b,n,s}).$$
\end{lemma}

\begin{proof}
We know that the function field $\mathbb{F}_{q^{2n}}(u,z)$ is defined by $z^{m/s}=(u^{q^2-1}-1)/u^{q-1}$, has genus $(m/s-1)(q^2-1)/2$, and the zero of $u$ is totally ramified in $\mathbb{F}_{q^{2n}}(u,z)/\mathbb{F}_{q^{2n}}(u)$.
Then, using \cite[Proposition 3.7.10]{stichtenoth}, $\Fqn(u, v,z)/\Fqn(u, z)$ is a extension of genus
    \begin{align*}
        g(\tilde{\mathcal{X}}_{a,b,n,s}) &= \frac{q}{p^b} \cdot \frac{(m/s-1)(q^2-1)}{2} + \frac{q-p^b}{2p^b}(-2 + m(q + 1)/s + 1) \\
        &= \frac{1}{2sp^b} \left( q^{n+2} - p^bq^{n} - sq^3 + q^2 + (s-1)p^b \right).
    \end{align*}
\end{proof}

In the rest of this section, we determine some Weierstrass semigroups at places of $\tilde{\mathcal{X}}_{a,b,n,s}$ (Section \ref{subsec:weierstrass}), and compute the automorphism group of $\tilde{\mathcal{X}}_{a,b,n,s}$ (Section \ref{subsec:autX}), which will allow us to distinguish birationally $\tilde{\mathcal{X}}_{a,b,n,s}$ from $\mathcal{X}_{a,b,n,s}$.


\subsection{Distinguishing places of $\tilde{\mathcal{X}}_{a,b,n,s}$}\label{subsec:weierstrass}

We start by computing the Weierstrass semigroup at several rational places of  $\tilde{\mathcal{X}}_{a,a,n,s}$, i.e. $b=a$. In this case,
$$ \tilde{\mathcal{X}}_{a,a,n,s}=\mathbb{F}_{q^{2n}}(\rho,z),\qquad z^{m/s}=(\rho^{q^2-1}-1)/\rho^{q-1},\qquad g(\tilde{\mathcal{X}}_{a,a,n,s})=\frac{(q^2-1)(m/s-1)}{2}. $$



Let $\tilde{Q}_\infty^1$ be the unique place of $\tilde{\mathcal{X}}_{a,a,n,s}$ lying below the place $\tilde{P}_\infty^1$ of $\tilde{\mathcal{Y}}_{n,s}$; then $\tilde{Q}_\infty^1$ is the unique zero of $\rho$.
Moreover, let $\tilde{Q}_\infty^2$ be the place of $\tilde{\mathcal{X}}_{a,a,n,s}$ lying below the $\tilde{E}_q$-orbit $\{\tilde{P}_\infty^2,\dots,\tilde{P}_\infty^{q+1}\}$; then $\tilde{Q}_\infty^2$ is the unique pole of $\rho$. More precisely,
$$(\rho)_{\tilde{\mathcal{X}}_{a,a,n,s}}=\frac{m}{s} \left( \tilde{Q}_\infty^1-\tilde{Q}_\infty^2\right).$$
Finally, the $q^3-q$ places of $\tilde{\mathcal{Y}}_{n,s}$ contained in $O$ restrict to the $q^2-1$ zeroes of $\rho^{q^2-1}-1$ in $\tilde{\mathcal{\mathcal{X}}}_{a,a,n,s}$, which have multiplicity $m/s$ and will be indicated with $\tilde{Q}_\alpha$ for $\alpha \in \mathbb{F}_{q^2}^*$, so that
$$(\rho-\alpha)_{\tilde{\mathcal{X}}_{a,a,n,s}}=\frac{m}{s} \left( \tilde{Q}_\alpha-\tilde{Q}_\infty^2\right) \quad \text{ and } \quad ((\rho-\alpha)/\rho)_{\tilde{\mathcal{X}}_{a,a,n,s}}=\frac{m}{s} \left( \tilde{Q}_\alpha-\tilde{Q}_\infty^1\right).$$
Then
$$(z)_{\tilde{\mathcal{X}}_{a,a,n,s}}=\sum_{\alpha \in \mathbb{F}_{q^2}^*} \tilde{Q}_\alpha-(q-1)\tilde{Q}_\infty^1-(q^2-q) \tilde{Q}_\infty^2,$$
$$(\rho^k z^\ell)_{\tilde{\mathcal{X}}_{a,a,n,s}}=\ell \sum_{\alpha \in \mathbb{F}_{q^2}^*} \tilde{Q}_\alpha+\left(k\frac{m}{s}-\ell(q-1)\right) \tilde{Q}_\infty^1+\left(-k\frac{m}{s}-\ell(q^2-q)\right) \tilde{Q}_\infty^2.$$
This implies the following inclusion of semigroups:
$$H_1:=\left\langle \frac{m}{s},\;a_\ell \;\colon\; \ell=1,\dots,\frac{m}{s}-1 \right\rangle \subseteq H(\tilde{Q}_\infty^1),$$ 
$$
H_2:=\left\langle \frac{m}{s},\;b_\ell \;\colon\; \ell=1,\dots,\frac{m}{s}-1 \right\rangle \subseteq H(\tilde{Q}_\infty^2),
$$
where
$$
a_\ell=\ell (q-1)+ \left\lceil \frac{\ell (q^2-q)}{m/s}\right\rceil \frac{m}{s},\qquad b_\ell=\ell (q^2-q)+\left\lceil \frac{\ell (q-1)}{m/s}\right\rceil \frac{m}{s}.
$$
Define
\begin{equation}\label{eq:apery}
A_1:=\left\{ 0,\;a_\ell \;\colon\; \ell=1,\dots,\frac{m}{s}-1 \right\},\qquad
A_2:=\left\{ 0,\;b_\ell \;\colon\; \ell=1,\dots,\frac{m}{s}-1 \right\}.
\end{equation}
Since $\gcd(q-1,m/s)=\gcd(q^2-q,m/s)=1$, $A_i$ ($i=1,2$) is a complete system modulo $m/s$, and its elements are minimal in $H_i$ with this property; that is, $A_i$ the Ap\'ery set ${\rm Ap}(H_i,m/s)$ of $m/s$ in $H_i$.
Then, by Selmer's formula, the number of gaps in $H_1$ and $H_2$ is
$$
g(H_1)=\frac{1}{m/s}\left(\sum_{\ell=1}^{m/s-1}a_\ell\right)-\frac{m/s-1}{2}=\frac{(q^2-1)(m/s-1)}{2}=g(\tilde{\mathcal{X}}_{a,a,n,s}),
$$
$$
g(H_2)=\frac{1}{m/s}\left(\sum_{\ell=1}^{m/s-1}b_\ell\right)-\frac{m/s-1}{2}=\frac{(q^2-1)(m/s-1)}{2}=g(\tilde{\mathcal{X}}_{a,a,n,s}).
$$
This means that $H_1=H(\tilde{Q}_\infty^1)$ and $H_2=H(\tilde{Q}_\infty^2)$.
Now, considering functions of the form 
$$\left(\frac{\rho-\alpha}{\rho}\right)^{k_1} \frac{1}{\rho^{k_2}} z^\ell,$$
one similarly sees that
$$
H_{(\alpha)}:=\left\langle \frac{m}{s},\;-\ell+\left(\left\lceil \frac{\ell (q-1)}{m/s}\right\rceil+\left\lceil \frac{\ell (q^2-q)}{m/s}\right\rceil \right) \frac{m}{s} \;\colon\; \ell=1,\dots,\frac{m}{s}-1 \right\rangle \subseteq H(\tilde{Q}_\alpha).
$$
By the same argument as above, the Ap\'ery set of $m/s$ in $H_{(\alpha)}$ is given by its generators above, after replacing $m/s$ with zero; then the Selmer's formula yields $H_{(\alpha)}=H(\tilde{Q}_{\alpha})$. 
We have shown the following result.

\begin{theorem}\label{thm:semigroupsXaa}
The semigroups at the places $\tilde{Q}_\infty^{1}$, $\tilde{Q}_\infty^{2}$ and $\tilde{Q}_\alpha$ of $\tilde{\mathcal{X}}_{a,a,n,s}$ are as follows.
$$H(\tilde{Q}_\infty^1)=\left\langle \frac{m}{s},\;\ell (q-1)+ \left\lceil \frac{\ell (q^2-q)}{m/s}\right\rceil \frac{m}{s} \;\colon\; \ell=1,\dots,\frac{m}{s}-1 \right\rangle,$$
$$
H(\tilde{Q}_\infty^2)=\left\langle \frac{m}{s},\;\ell (q^2-q)+\left\lceil \frac{\ell (q-1)}{m/s}\right\rceil \frac{m}{s} \;\colon\; \ell=1,\dots,\frac{m}{s}-1 \right\rangle,
$$
$$
H(\tilde{Q}_\alpha)=\left\langle \frac{m}{s},\;-\ell+\left(\left\lceil \frac{\ell (q-1)}{m/s}\right\rceil+\left\lceil \frac{\ell (q^2-q)}{m/s}\right\rceil \right) \frac{m}{s} \;\colon\; \ell=1,\dots,\frac{m}{s}-1 \right\rangle.
$$
\end{theorem}

The description of the semigroups distinguishes the corresponding places in most cases.

\begin{corollary}\label{cor:semigroupsXaa}
If $m/s$ divides $q^2-q+1$, then $H(\tilde{Q}_\infty^2)=H(\tilde{Q}_\alpha)$.
\end{corollary}
\begin{proof}
Let $\ell \in \{1,\dots,m/s-1\}$. Then 
\begin{eqnarray*}
-\ell+\left(\left\lceil \frac{\ell (q-1)}{m/s}\right\rceil+\left\lceil \frac{\ell (q^2-q)}{m/s}\right\rceil \right) \frac{m}{s} & = & -\ell+\left(\left\lceil \frac{\ell (q-1)}{m/s}\right\rceil+\frac{\ell (q^2-q+1)}{m/s}+ \left\lceil \frac{-\ell}{m/s}\right\rceil\right) \frac{m}{s}\\
& = & \ell (q^2-q)+\left\lceil \frac{\ell (q-1)}{m/s}\right\rceil \frac{m}{s}.
\end{eqnarray*}
\end{proof}

\begin{theorem}\label{thm:semigroupsXaa_compare}
If $m/s>1$, then $H(\tilde{Q}_\infty^1) \neq H(\tilde{Q}_\infty^2)$ and $H(\tilde{Q}_\infty^1) \neq H(\tilde{Q}_\alpha)$. If $m/s$ does not divide $q^2-q+1$, then $H(\tilde{Q}_\infty^2) \neq H(\tilde{Q}_\alpha)$. 
\end{theorem}
\begin{proof}
{\bf Step 1: $H(\tilde{Q}_\infty^1) \neq H(\tilde{Q}_\infty^2)$.} 
Suppose that $H(\tilde{Q}_\infty^1) = H(\tilde{Q}_\infty^2)$. Using the Ap\'ery sets \eqref{eq:apery}, for any $\ell_1 \in \{1,\dots,m/s-1\}$ there exists a unique $\ell_2 \in \{1,\dots,m/s-1\}$ such that 
\begin{equation}\label{eq:semgr1}
\ell_1 (q-1)+\left\lceil \frac{\ell_1 (q^2-q)}{m/s}\right\rceil \frac{m}{s} =\ell_2 (q^2-q)+\left\lceil \frac{\ell_2 (q-1)}{m/s}\right\rceil \frac{m}{s}.
\end{equation}
In particular, there exists an integer $\mu$ such that $\ell_1=\ell_2 q+\mu\frac{m}{s}$, so that Equation \eqref{eq:semgr1} implies
$$\mu(q-1)+\left\lceil \frac{\ell_2 q(q^2-q)}{m/s}\right\rceil +\mu(q^2-q)=\left\lceil \frac{\ell_2 (q-1)}{m/s}\right\rceil.$$
Then, after defining $\gamma=\ell_2(q-1)/(m/s)$, we obtain
\begin{equation}\label{eq:semgr2}
q^2\lceil\gamma\rceil - \lceil q^2\gamma\rceil\equiv0\pmod{q^2-1}.
\end{equation}

Now, from $q^2\gamma\notin\mathbb{Z}$ it follows that $-\lceil q^2\gamma\rceil=\lceil q^2\gamma\rceil-1$. Writing $r=\lceil\gamma\rceil-\gamma$, this implies
\begin{equation}\label{eq:semigroup}
q^2\lceil\gamma\rceil-\lceil q^2\gamma\rceil = \lceil q^2 r \rceil-1.
\end{equation}

Since $0<r<1$, we have $0\leq\lceil q^2r\rceil-1\leq q^2-1$. By Equations \eqref{eq:semgr2} and \eqref{eq:semigroup}, this yields
\begin{equation}\label{eq:or}\lceil q^2 r\rceil=1\quad\mbox{or}\quad\lceil q^2 r\rceil=q^2.\end{equation}


Since $\gcd(q-1,m/s)=1$, the quantity $r=\lceil \ell_2 (q-1)/(m/s)\rceil-\ell_2 (q-1)/(m/s)$ takes all the values $i/(m/s)$ with $1 \le i \le m/s-1$ as $\ell_2$ runs in $\{1,\dots,m/s-1\}$. 

Firstly, choose $\ell_2$ such that $r=1/(m/s)$. Since $m/s \ge 3$, the case $\lceil q^2 r \rceil =q^2$ does not hold, and hence by Equation \eqref{eq:or} $\lceil q^2/(m/s) \rceil =1$, which implies $q^2 \le m/s$.

Secondly, choose $\ell_2$ such that  $r=2/(m/s)$. Since we proved $q^2\leq m/s$ above, we have $\lceil 2q^2/(m/s) \rceil \le 2<q^2$ and hence the case $\lceil q^2 r\rceil=q^2$ cannot hold. Then by Equation \eqref{eq:or} $\lceil q^2 r\rceil =1$, which implies $2q^2 \le m/s$.

Continuing in this way, one proves $\lceil i q^2/(m/s)\rceil\leq 2$ and hence $i q^2\leq(m/s)$  for every $i=1,\ldots,m/s-1$. In particular $(m/s-1)q^2 \le m/s$, so that $m/s \le q^2/(q^2-1)$. This provides a contradiction to $m/s\geq2$. 

{\bf Step 2: $H(\tilde{Q}_\infty^1) \neq H(\tilde{Q}_\alpha)$.} 
Suppose $H(\tilde{Q}_\infty^1) = H(\tilde{Q}_\alpha)$. By the equality of the Ap\'ery sets of $m/s$, for every $\ell_1 \in \{1,\dots,m/s-1\}$ there exists a unique $\ell_2 \in \{1,\dots,m/s-1\}$ such that 
\begin{equation}\label{eq:semgr4}
\ell_1 (q-1)+\left\lceil \frac{\ell_1 (q^2-q)}{m/s}\right\rceil \frac{m}{s} =-\ell_2+\left(\left\lceil \frac{\ell_2 (q-1)}{m/s}\right\rceil+\left\lceil \frac{\ell_2 (q^2-q)}{m/s}\right\rceil \right) \frac{m}{s}.
\end{equation}
In particular there exists an integer $\mu$ with $\ell_2=-\ell_1 (q-1)+\mu\frac{m}{s}$, so that, by Equation \eqref{eq:semgr4},
$$\left\lceil \frac{\ell_1 (q^2-q)}{m/s}\right\rceil=-\mu+\left\lceil \frac{-\ell_1 (q-1)^2}{m/s}\right\rceil+\mu(q-1) +\left\lceil \frac{-\ell_1 q(q-1)^2}{m/s}\right\rceil+\mu(q^2-q),$$
whence, using that $\lceil-\delta\rceil=1-\lceil\delta\rceil$ whenever $\delta$ is not integer, we obtain
\begin{equation}\label{eq:semgr5}
\left\lceil \frac{\ell_1 (q^2-q)}{m/s}\right\rceil+ \left\lceil \frac{\ell_1 (q-1)^2}{m/s}\right\rceil +\left\lceil \frac{\ell_1 q(q-1)^2}{m/s}\right\rceil \equiv 2 \pmod{q^2-2}.
\end{equation}
Since the left-hand side of Equation \eqref{eq:semgr5} is at least three, we conclude that
$$\left\lceil \frac{\ell_1 (q^2-q)}{m/s}\right\rceil+ \left\lceil \frac{\ell_1 (q-1)^2}{m/s}\right\rceil +\left\lceil \frac{\ell_1 q(q-1)^2}{m/s}\right\rceil \ge q^2$$
for all $\ell_1 \in \{1,\dots,m/s-1\}$. 
If $m/s \ge q$, then we get a contradiction using $\ell_1=1$ :
$$\left\lceil \frac{\ell_1 (q^2-q)}{m/s}\right\rceil+ \left\lceil \frac{\ell_1 (q-1)^2}{m/s}\right\rceil +\left\lceil \frac{\ell_1 q(q-1)^2}{m/s}\right\rceil \le \left\lceil \frac{q^2-q}{q}\right\rceil+ \left\lceil \frac{(q-1)^2}{q}\right\rceil +\left\lceil \frac{q(q-1)^2}{q}\right\rceil=q^2-1.$$
Hence we can assume that $m/s \le q-1$. 
Now define 
\begin{equation}\label{eq:def_f}  f(\ell):=\left\lceil \frac{\ell (q-1)}{m/s}\right\rceil+\left\lceil \frac{\ell (q^2-q)}{m/s}\right\rceil-1.\end{equation}
Since $(q-1)/(m/s)\geq1$, the function $f(\ell)$ is strictly monotone in $\ell$, because
\begin{equation}\label{eq:monot}
f(\ell+1)=\left\lceil \frac{\ell (q-1)}{m/s}+\frac{q-1}{m/s}\right\rceil+\left\lceil \frac{\ell (q^2-q)}{m/s}+q\frac{q-1}{m/s}\right\rceil-1\geq 
f(\ell) + q+1.
\end{equation}

In terms of the function $f(\ell)$, Equation \eqref{eq:semgr4} reads, after some computation,
\begin{equation}\label{eq:fmono}
f(\ell_2)-f(\ell_1)=\frac{\ell_2}{m/s}-\left( \left\lceil\frac{\ell_1(q-1)}{m/s}\right\rceil - \frac{\ell_1(q-1)}{m/s}. \right)\end{equation}
If $\ell_2>\ell_1$, then Equations \eqref{eq:monot} and \eqref{eq:fmono} yield the following contradiction:
$$ q+1 \leq |f(\ell_2)-f(\ell_1)|\leq \ell_2/(m/s)+1<2. $$



Therefore $\ell_2=\ell_1$, and Equation \eqref{eq:fmono} implies
$$\ell_1q=\left\lceil\frac{\ell_1 (q-1)}{m/s}\right\rceil \frac{m}{s}.$$
Note that this holds for all $\ell_1 \in \{1,\dots,m/s-1\}$.
For $\ell_1=1$, we obtain $q=\left\lceil\frac{q-1}{m/s}\right\rceil\cdot\frac{m}{s}$. Then $m/s$ divides $q$; being coprime, this means $m/s=1$. Thus, $q=\lceil(q-1)/1\rceil\cdot1$, a contradiction.

{\bf Step 3: $H(\tilde{Q}_\infty^2) \neq H(\tilde{Q}_\alpha)$.} 
By arguing similarly to Step 2, one shows
\begin{equation*}
\left\lceil \frac{\ell_1 (q-1)}{m/s}\right\rceil+ \left\lceil \frac{\ell_1 q(q-1)^2}{m/s}\right\rceil +\left\lceil \frac{\ell_1 q^2(q-1)^2}{m/s}\right\rceil \equiv 2 \pmod{q^2-2},
\end{equation*}
whence
\begin{equation}\label{eq:ineq}\left\lceil \frac{\ell_1 (q-1)}{m/s}\right\rceil+ \left\lceil \frac{\ell_1 q(q-1)^2}{m/s}\right\rceil +\left\lceil \frac{\ell_1 q^2(q-1)^2}{m/s}\right\rceil \ge q^2.\end{equation}
By direct computation, \eqref{eq:ineq} implies $m/s \le q^2-q+2$. Since $m$ is always odd, the case $m/s=q^2-q+2$ cannot occur. Moreover, the case $m/s=q^2-q+1$ is excluded by the assumption of Step 3 that $m/s\nmid(q^2-q+1)$. Hence we can assume that $m/s \le q^2-q$. Under this assumption, it is easily seen that the function $f(\ell)$ defined as in Equation \eqref{eq:def_f} is strictly monotonous.
As in Step 2, the equality of the Ap\'ery sets of $m/s$ in $H(\tilde{Q}_\infty^2)$ and $H(\tilde{Q}_\alpha)$ can be written in terms of the function $f(\ell)$, and as in Step 2 we can conclude that
$$
\ell(q^2-q+1)=\left\lceil\frac{\ell (q^2-q)}{m/s}\right\rceil \frac{m}{s}
$$
for all $\ell \in \{1,\dots,m/s-1\}$. For $\ell=m/s-1$, this shows that $m/s$ divides $(m/s-1)(q^2-q+1)$ and hence $m/s\mid(q^2-q+1)$, a contradiction to the assumption of Step 3.
\end{proof}

We turn our attention to 
$\tilde{\mathcal{X}}_{a,b,n,s}$ in the case $b<a$.
Consider the function field $\mathbb{F}_{q^{2n}}(\rho,v)$ defined by 
$\sum_{j=0}^{a/b-1} v^{p^{jb}}=c_0(\rho^{q+1}+1)/\rho^{q+1}$;
the first step is to compute the semigroups at the places of $\mathbb{F}_{q^{2n}}(\rho,v)$ lying under the places of $\tilde{\mathcal{Y}}_{n,s}$ in $O\cup O_\infty$. 
We denote by $\tilde{R}_\infty^1,\dots,\tilde{R}_\infty^{q/p^b+1}$ the places of $\mathbb{F}_{q^{2n}}(\rho,v)$ lying under $O_\infty$; in particular, by $\tilde{R}_\infty^1$ the restriction of $\tilde{P}_\infty^1$, and by $\tilde{R}_\infty^2,\dots,\tilde{R}_\infty^{q/p^b+1}$ the restrictions of $\tilde{P}_\infty^2,\dots,\tilde{P}_\infty^{q+1}$.
Similarly, for any $\alpha\in\mathbb{F}_{q^2}\setminus\{0\}$ we denote by $\tilde{R}_\alpha$ the restriction of the zero $\tilde{P}_\alpha$ of $\rho-\alpha$ in $\tilde{\mathcal{Y}}_{n,s}$.
Note that all such places $\tilde{R}_\infty^j$ and $\tilde{R}_\alpha$ are totally ramified with ramification index $m/s$ in the extension $\tilde{\mathcal{X}}_{a,b,n,s}/\mathbb{F}_{q^{2n}}(\rho,v)$.

\begin{proposition}\label{thm:semigroupsnurho}
We have $$H(\tilde{R}_\infty^1)=\langle q/p^b,q+1\rangle,$$ and for $j \ge 2$ and $\alpha \in \mathbb{F}_{q^2}\setminus\{0\}$ : 
$$H(\tilde{R}_\infty^j)=H(\tilde{R}_\alpha)=\langle q-p^b+1,\dots,q,q+1\rangle.$$
\end{proposition}
\begin{proof}
We have the following principal divisors: 
$$(\rho)_{\mathbb{F}_{q^{2n}}(\rho,v)}=\frac{q}{p^b}\tilde{R}_\infty^1 - \sum_{j=2}^{q/p^b+1}\tilde{R}_\infty^j,\qquad (v)_{\mathbb{F}_{q^{2n}}(\rho,v)}=-(q+1)\tilde{R}_\infty^1 + \sum_{\alpha^{q+1}=-1}\tilde{R}_\alpha.$$
Through the functions $1/\rho$ and $v$ we have $H(\tilde{R}_\infty^1) \supseteq \langle q/p^b,q+1\rangle$; since $\langle q/p^b,q+1\rangle$ contains precisely $(q/p^b-1)q/2=g(\mathbb{F}_{q^{2n}}(\rho,v))$ many gaps, it holds $H(\tilde{R}_\infty^1) = \langle q/p^b,q+1\rangle$.

Recalling that $\rho,v\in\mathbb{F}_{q^2}(x,y)$ with $y^{q+1}=x^{q+1}-1$, we have that $\mathbb{F}_{q^2}(\rho,v)$ is $\mathbb{F}_{q^2}$-maximal, as a subfield of the $\mathbb{F}_{q^2}$-maximal Hermitian function field.
Then we can apply the fundamental equality for maximal function fields, see \cite[Theorem 9.79 and Equation (10.8)]{HKT}, to conclude that for any $j=2,\ldots,q/p^b+1$ there exists a function $f_j\in\mathbb{F}_{q^2}(\rho,v)$ with
$$(f_j)_{\mathbb{F}_{q^{2n}}(\rho,v)}=-(q+1)\tilde{R}_\infty^1+(q+1)\tilde{R}_\infty^j.$$ 

Now fix $j\geq2$ and consider the functions $(\rho^k f_j)^{-1}$ for $k=0,1,\dots,p^b$. Then $(\rho^kf_j)^{-1}$ has a pole only at $\tilde{R}_\infty^j$, with multiplicity $q+1-k$. Hence $H(\tilde{R}_\infty^j) \supseteq \langle q-p^b+1,\dots,q,q+1\rangle$. By direct computation, the semigroup on the right-hand side has $(q/p^b-1)q/2$ many gaps and thus the equality $H(\tilde{R}_\infty^j)= \langle q-p^b+1,\dots,q,q+1\rangle$ holds.

Finally, let $\tilde{R}_\alpha$ be a zero of $\rho^{q^2-1}-1$ in $\mathbb{F}_{q^{2n}}(\rho,v)$. Then, for some effective divisor $E$, 
$$((\rho-\alpha)/\rho)_{\mathbb{F}_{q^{2n}}(\rho,v)}=\tilde{R}_\alpha+E-\frac{q}{p^b}\tilde{R}_\infty^1$$
Moreover, by the fundamental equality, there exists a function $f_\alpha\in\mathbb{F}_{q^2}(\rho,v)$ such that
$$(f_\alpha)_{\mathbb{F}_{q^{2n}}(\rho,v)}=-(q+1)\tilde{R}_\alpha+(q+1)\tilde{R}_\infty^1.$$
For $k=0,1,\ldots,p^b$, the function $((\rho-\alpha)/\rho)^k f_\alpha$ has a pole only at $\tilde{R}_\alpha$, with multiplicity $q+1-k$. This shows that  $H(\tilde{R}_\alpha)=H(\tilde{R}_\infty^j)$ with $j \ge 2$.
\end{proof}

We introduce the following notation for the places of $\tilde{\mathcal{X}}_{a,b,n,s}$:
$\overline{P}_\alpha:=\tilde{P}_\alpha \cap \tilde{\mathcal{X}}_{a,b,n,s}$, 
$\overline{P}_\infty^1:=\tilde{P}_\infty^1 \cap \tilde{\mathcal{X}}_{a,b,n,s}$ 
and $\left\{\overline{P}_\infty^j\colon j=2,\ldots,q/p^b+1\right\}=\left\{\tilde{P}_\infty^j \cap \tilde{\mathcal{X}}_{a,b,n,s}\colon j=2,\ldots,q+1\right\}$.
The following picture collects the notation for some places in the function fields studied so far.
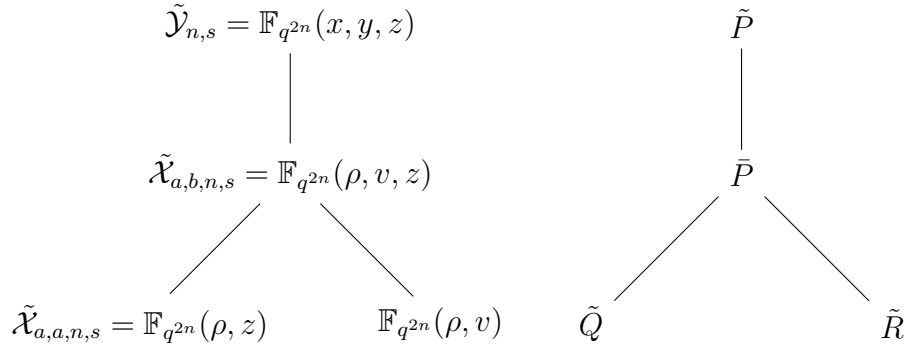
\begin{center}
\begin{tikzpicture}
    \node (A) at (0,0) {$\tilde{\mathcal{Y}}_{n,s}=\mathbb{F}_{q^{2n}}(x,y,z)$};
    \node (B) at (6,0) {$\tilde{P}$};
    \node (C) at (0,-2) {$\tilde{\mathcal{X}}_{a,b,n,s}=\mathbb{F}_{q^{2n}}(\rho,v,z)$};
    \node (D) at (6,-2) {$\bar{P}$};
      \node (E) at (-2,-4) {$\tilde{\mathcal{X}}_{a,a,n,s}=\mathbb{F}_{q^{2n}}(\rho,z)$};
         \node (G) at (4,-4) {$\tilde{Q}$};
          \node (H) at (8,-4) {$\tilde{R}$};
\node (F) at (2,-4) {$\mathbb{F}_{q^{2n}}(\rho,v)$};
    \draw (A) -- (C);
    \draw (B) -- (D);
    \draw (C) -- (E);
      \draw (C) -- (F);
       \draw (D) -- (G);
        \draw (D) -- (H);
\end{tikzpicture}
\end{center}

\begin{proposition}\label{prop:Xab1}
Let $1 \le b \le a-1$. Then
$$
H(\overline{P}_\infty^1) \neq H(\overline{P}_\alpha)\quad\mbox{for}\quad \alpha\in\mathbb{F}_{q^2}\setminus\{0\},\qquad H(\overline{P}_\infty^1) \neq H(\overline{P}_\infty^j)\quad\mbox{for}\quad j=2,\dots,q/p^b+1.
$$
\end{proposition}
\begin{proof}
We know from Theorem \ref{thm:semigroupsXaa} that $m/s \in H(\tilde{Q}_\infty^1)$. Since $\tilde{Q}_\infty^1$ is totally ramified in the extension $\tilde{\mathcal{X}}_{a,b,n,s}/\tilde{\mathcal{X}}_{a,a,n,s}$, this implies that $(m/s) (q/p^b) \in H(\overline{P}_\infty^1)$.
In order to complete the proof, we claim that $(m/s) (q/p^b)$ is a gap number at the places $\tilde{P}_\alpha$ and $\overline{P}_\infty^j$ with $j \ge 2$.

The quadratic polynomial $h(T):=T^2-(q+1)T+q$ has roots $1$ and $q$; since $1<p^b<q$, this implies $h(p^b)<0$, which is equivalent to $q/p^b<q-p^b+1$. Thus, by Proposition \ref{thm:semigroupsnurho},  $q/p^b$ is a gap number at the places $\tilde{R}_\alpha$ and $\tilde{R}_\infty^j$, $j \ge 2$.
By Lemma \ref{gapsdiff},
there exist regular differentials $\omega_j$, $\omega_\alpha$ such that $v_{\tilde{R}_\infty^j}(\omega_j)=v_{\tilde{R}_\alpha}(\omega_\alpha)=q/p^b-1$.
Since $\tilde{R}_\alpha$ and $\tilde{R}_\infty^j$, $j \ge 2$, are totally ramified in $\tilde{\mathcal{X}}_{a,b,n,s}/\mathbb{F}_{q^{2n}}(\rho,v)$ with tame ramification index $m/s$, we have $v_{\overline{P}_\alpha}(\omega_\alpha)=v_{\overline{P}_\infty^j}(\omega_j)=(q/p^b) (m/s)-1.$ Hence $(q/p^b) (m/s)$ is a gap at $\overline{P}_\alpha$ and $\overline{P}_\infty^j$, $j \ge 2$. 
\end{proof} 

Finally, we compare the semigroups at the places $\overline{P}_\alpha$ and $\overline{P}_\infty^j$, $j \ge 2$. To this aim, we use the following divisors, of a differential and some functions; 
for the divisor \eqref{eq:xabns_div_diff} of the nonzero differential, recall that its degree is $2g(\tilde{\mathcal{X}}_{a,b,n,s})-2$ and use Lemma \ref{genXtilda}.
\begin{multline}\label{eq:xabns_div_diff}
\left(\frac{1}{z^{m/s-1}} d\left(\frac{1}{\rho}\right)\right)_{\tilde{\mathcal{X}}_{a,b,n,s}} = \\ \left(\frac{m}{s}\left(\frac{2q^2}{p^b}-q-\frac{q}{p^b}-1\right)-\frac{q^2-q}{p^b}-1 \right)\overline{P}_\infty^1+\left(\frac{m}{s}-1\right)(q^2-q+1)\sum_{j \ge 2}\overline{P}_\infty^j,
\end{multline}
\begin{equation}\label{eq:xabns_div_z}
\left(z\right)_{\tilde{\mathcal{X}}_{a,b,n,s}}=\overline{P}_\alpha + {E}_\alpha^\prime-\frac{q^2-q}{p^b}\overline{P}_\infty^1-(q^2-q)\sum_{j \ge 2}\overline{P}_\infty^j,
\end{equation}
\begin{equation}\label{eq:xabns_div_rho}
\left(\rho\right)_{\tilde{\mathcal{X}}_{a,b,n,s}}=\frac{m}{s} \frac{q}{p^b}\overline{P}_\infty^1-\frac{m}{s}\sum_{j \ge 2}\overline{P}_\infty^j,
\end{equation}
\begin{equation}\label{eq:xabns_div_rhoalpha}
\left(\frac{\rho-\alpha}{\rho}\right)_{\tilde{\mathcal{X}}_{a,b,n,s}}=\frac{m}{s}\overline{P}_\alpha +E_\alpha-\frac{m}{s} \frac{q}{p^b}\overline{P}_\infty^1,
\end{equation}
for any $\alpha\in\mathbb{F}_{q^2}\setminus\{0\}$, where $E_\alpha^\prime$ and $E_\alpha$ are effective divisors, whose supports do not contain any of the places $\overline{P}_\infty^1,\dots,\overline{P}_\infty^{q/p^b+1},\overline{P}_\alpha$.
Further, the fundamental equality applied to the $\mathbb{F}_{q^2}$-maximal function field $\mathbb{F}_{q^2}(\rho,v)$ provides functions $f_j,f_\alpha\in\mathbb{F}_{q^2}(\rho,v)$ such that
\begin{equation}\label{eq:xabns_div_fundeq}
\left(f_j\right)_{\tilde{\mathcal{X}}_{a,b,n,s}}=(q+1)\frac{m}{s}(\overline{P}_\infty^1-\overline{P}_\infty^j) \quad \text{and} \quad \left(f_\alpha\right)_{\tilde{\mathcal{X}}_{a,b,n,s}}=(q+1)\frac{m}{s}(\overline{P}_\alpha-\overline{P}_\infty^1).
\end{equation}

\begin{proposition}\label{prop:Xab2}
Let $\alpha\in\mathbb{F}_{q^2}\setminus\{0\}$ and $j\in\{2,\ldots,q/p^b+1\}$.
If 
$m/s$ does not divide $q^2-q+1$, then $H(\overline{P}_\alpha) \neq H(\overline{P}_\infty^j)$.
\end{proposition}

\begin{proof}
Define $$A:=\left\lceil \frac{(q^2-q)/p^b}{m/s} \right\rceil,  \qquad B:=\left\lceil \frac{q^2-q}{m/s} \right\rceil.$$
By Equations \eqref{eq:xabns_div_z} and \eqref{eq:xabns_div_fundeq}, $\overline{P}_\infty^j$ is the only pole of $(f_j/\rho^{p^b})^Az$, and hence its order $(m/s)A(q+1-p^b)+q^2-q$ is in $H(\overline{P}_\infty^j)$. Thus, by Lemma \ref{gapsdiff}, it is enough to provide a regular differential $\omega$ such that $v_{\overline{P}_\alpha}(\omega)=(m/s)A(q+1-p^b)+q^2-q-1$.
Consider the differential
$$\eta_\alpha:=f_\alpha^{A-1} \left(\frac{\rho-\alpha}{\rho}\right)^{q-Ap^b+B} z^{q^2-q-(B-1)\frac{m}{s}-1} \frac{1}{z^{m/s-1}} d\left(\frac{1}{\rho}\right).$$
A lengthy but straightforward calculation using Equations \eqref{eq:xabns_div_fundeq}, \eqref{eq:xabns_div_rhoalpha}, \eqref{eq:xabns_div_z} and \eqref{eq:xabns_div_diff} shows
\begin{multline*}
\left(\eta_\alpha\right)_{\tilde{\mathcal{X}}_{a,b,n,s}}=\left( \frac{m}{s}A(q+1-p^b)+q^2-q-1 \right)\overline{P}_\alpha + E_\alpha^{\prime\prime}+\\
\left(
\frac{1}{p^b}\left((q^2-q)\left(B\frac{m}{s}-q^2+q \right)\right)-B\frac{q}{p^b}\frac{m}{s}-A\frac{m}{s}-1\right)\overline{P}_\infty^1+\\
\left((q^2-q)\left(B\frac{m}{s}-q^2+q \right)+\frac{m}{s}-1\right)\sum_{j \ge 2} \overline{P}_\infty^j,
\end{multline*}
where $E_\alpha^{\prime\prime}$ is an effective divisor whose support does not contain any of the places $\overline{P}_\alpha$, $\overline{P}_\infty^1,\dots,\overline{P}_\infty^{q/p^b+1}$. Let $C:=Bm/s-q^2+q$. Then $C$ is the smallest positive integer such that $m/s$ divides $q^2-q+C$. In particular $2 \le C \le m/s-1$, since $\gcd(m,q^2-q)=1$ and $m/s$ does not divide $q^2-q+1$ by hypothesis.
With this notation in place, we obtain 
\begin{multline*}
\left(\eta_\alpha\right)_{\tilde{\mathcal{X}}_{a,b,n,s}}=\left( \frac{m}{s}A(q+1-p^b)+q^2-q-1 \right)\overline{P}_\alpha + E_\alpha^{\prime\prime}+\\
\left(
\frac{1}{p^b}\left(B\frac{m}{s}-C\right)C-B\frac{q}{p^b}\frac{m}{s}-A\frac{m}{s}-1\right)\overline{P}_\infty^1+\left(\left(B\frac{m}{s}-C\right)C+\frac{m}{s}-1\right)\sum_{j \ge 2} \overline{P}_\infty^j.
\end{multline*}
Define $\delta:=\lfloor ((m/s-C)C-1)/(m/s) \rfloor$ and the differential $\omega_\alpha:=\rho^{(B-1)C+1+\delta}\eta_\alpha$.
Some computation yields its divisor:
\begin{multline*}
\left(\omega_\alpha\right)_{\tilde{\mathcal{X}}_{a,b,n,s}}=\left( \frac{m}{s}A(q+1-p^b)+q^2-q-1 \right)\overline{P}_\alpha + E_\alpha^{\prime\prime}+\left(\left(\frac{m}{s}-C\right)C-1+\delta\frac{m}{s}\right)\sum_{j \ge 2} \overline{P}_\infty^j\\
+\left(
\frac{q}{p^b}\left(q^2C-q^2+q+C^2-C\frac{m}{s}-2C+\frac{m}{s}(\delta+1)\right)-A\frac{m}{s}-1\right)\overline{P}_\infty^1.
\end{multline*}

By direct computation, the valutation of $\left(\omega_\alpha\right)_{\tilde{\mathcal{X}}_{a,b,n,s}}$ is non-negative at $\overline{P}_{\alpha}$ and at $\overline{P}_{\infty}^j$ for $j\geq2$; in particular, the valuation $D:=(m/s-C)C-1+\delta m/s$ at $\overline{P}_\infty^j$ ($j\geq2$) satisfies $0\leq D\leq m/s-1$.
Then the prove is complete once we show that the valuation of $\left(\omega_\alpha\right)_{\tilde{\mathcal{X}}_{a,b,n,s}}$ at $\overline{P}_{\infty}^1$ is non-negative.
Using $A \frac{m}{s} \le \frac{q}{p^b}(q-1)+\frac{m}{s}-1$ and $D \le \frac{m}{s}-1$, this is implied by 
\[
\frac{q}{p^b}\left((q^2-2)C-q^2\right) +\left(\frac{q}{p^b}-1\right)\frac{m}{s} \ge \frac{q}{p^b}\left(\frac{m}{s}-1\right);
\]
equivalently, $\omega_{\alpha}$ is regular if
\begin{equation}\label{eq:diff_regular}
\frac{q}{p^b}\left((q^2-2)C-q^2+1\right) \ge \frac{m}{s}.
\end{equation}

In order to prove Equation \eqref{eq:diff_regular}, we distinguish two cases. Firstly, suppose $m/s\geq q^2-q$. 
In this case $A=B=1$ and $C=m/s-q^2+q$. Since $\gcd(m/s,q^2-q)=1$, $m/s\ne q^2-q+1$ by hypothesis and $m/s\ne q^2-q+2$ as $m$ is odd, we can assume $m/s \ge q^2-q+3$. Then
\begin{eqnarray*}
\frac{q}{p^b}\left((q^2-2)C-q^2+1\right) & = & \frac{q}{p^b}\left((q^2-2)\left(\frac{m}{s}-q^2+q\right)-q^2+1\right)\\
& = & \frac{q}{p^b}(q^2-2)\frac{m}{s}-\frac{q}{p^b}
\left((q^2-2)(q^2-q)+q^2-1\right)\\
& \ge &
\frac{m}{s}+\frac{q}{p^b}\left( (q^2-3)(q^2-q+3)-
\left((q^2-2)(q^2-q)+q^2-1\right)\right)\\
& = & \frac{m}{s}+\frac{q}{p^b}\left(q^2+q-8\right) \ge \frac{m}{s}.
\end{eqnarray*}

Secondly, suppose $m/s\leq q^2-q-1$.
Since $C \ge 2$, we obtain
$$
\frac{q}{p^b}\left((q^2-2)C-q^2+1\right) \ge  \frac{q}{p^b}\left((q^2-2)2-q^2+1\right)  = \frac{q}{p^b}\left(q^2-3\right)\ge\frac{q}{p^b}\left(q^2-q-1\right) \ge  \frac{m}{s}.
$$
\end{proof}

\subsection{The automorphism group of $\tilde{\mathcal{X}}_{a,b,n,s}$}\label{subsec:autX}

Our aim is to compute the full automorphism group of the function field $\tilde{\mathcal{X}}_{a,b,n,s}$, 
both for $a=b$ and for $b <a$. We will start by computing a subgroup $G$ of $\aut(\tilde{\mathcal{X}}_{a,b,n,s})$ 
and then we will show that $G=\aut(\tilde{\mathcal{X}}_{a,b,n,s})$ in most cases.
This is achieved by applying the results of Section \ref{subsec:weierstrass}, as places of $\tilde{\mathcal{X}}_{a,b,n,s}$ in the same $\aut(\tilde{\mathcal{X}}_{a,b,n,s})$-orbit have the same Weierstrass semigroup. 
Let 
\begin{equation} \label{eqG}
G:=\{\sigma_{\alpha,\beta,\gamma} \mid 
\alpha^{q+1} \in \mathbb{F}_{p^b}^* , \  \gamma^{m/s}=\alpha^{-(q-1)}, \ \beta^q\alpha+\beta\alpha^q=\alpha^{q+1}-1\ \},    
\end{equation}
where $c_0\in\mathbb{F}_{q^2}^*$ is fixed such that $c_0^{q}+c_0=0$, and $\sigma_{\alpha,\beta,\gamma}$ is defined by
\begin{equation}
\sigma_{\alpha,\beta,\gamma}: (u,v,z) \mapsto \bigg(\alpha u, \frac{v}{\alpha^{q+1}}+\frac{c_0\beta^q}{\alpha^q}-\bigg(\frac{c_0\beta^q}{\alpha^q}\bigg)^{p^b}, \gamma z \bigg).  \end{equation}
By direct checking, $\alpha\in\mathbb{F}_{q^2}$, $\beta\in\mathbb{F}_{q^2}$ and $\gamma\in\mathbb{F}_{q^{2n}}$ for any $\sigma_{\alpha,\beta,\gamma}\in G$.

We first check that $\sigma_{\alpha,\beta,\gamma}$ is in fact an automorphism of $\tilde{\mathcal{X}}_{a,b,n,s}$. To do so, note that
$$\sigma_{\alpha,\beta,\gamma}(z^{m/s})=\gamma^{m/s}z^{m/s}=\frac{1}{\alpha^{q-1}}\cdot \frac{u^{q^2-1}-1}{u^{q-1}}= \frac{(\alpha u)^{q^2-1}-1}{(\alpha u)^{q-1}}=\sigma_{\alpha,\beta,\gamma}\left( 
\frac{u^{q^2-1}-1}{u^{q-1}} \right),$$
This shows that the first equation of $\tilde{\mathcal{X}}_{a,b,n,s}$ is preserved.
For the second equation, write $p(v):=\sum_{i=0}^{f-1} v^{p^{ib}}$. Then
$$\sigma_{\alpha,\beta,\gamma}(p(v))=\sigma_{\alpha,\beta,\gamma}\bigg(\sum_{i=0}^{f-1} v^{p^{ib}}\bigg)=\sum_{i=0}^{f-1} \bigg(\frac{v}{\alpha^{q+1}}+\frac{c_0\beta^q}{\alpha^q}-\bigg(\frac{c_0\beta^q}{\alpha^q}\bigg)^{p^b}\bigg)^{p^{ib}}$$
$$=\frac{p(v)}{\alpha^{q+1}}+\sum_{i=0}^{f-1}\bigg( \frac{c_0\beta^q}{\alpha^q}\bigg)^{p^{ib}}-\sum_{i=1}^{f} \bigg( \frac{c_0\beta^q}{\alpha^q}\bigg)^{p^{ib}}=\frac{p(v)}{\alpha^{q+1}}+\frac{c_0\beta^q}{\alpha^q}-\bigg(\frac{c_0\beta^q}{\alpha^q}\bigg)^q$$
$$=\frac{p(v)}{\alpha^{q+1}}+\frac{c_0\beta^q}{\alpha^q}+\frac{c_0\beta}{\alpha}=\frac{p(v)}{\alpha^{q+1}}+c_0\bigg( \frac{\beta^q\alpha+\beta \alpha^q}{\alpha^{q+1}}\bigg)=\frac{1}{\alpha^{q+1}}c_0\frac{u^{q+1}+1}{u^{q+1}}+c_0\bigg( \frac{\alpha^{q+1}-1}{\alpha^{q+1}}\bigg)$$
$$
=c_0\frac{u^{q+1}+1+\alpha^{q+1}u^{q+1}-u^{q+1}}{\alpha^{q+1}u^{q+1}}=c_0\frac{(\alpha u)^{q+1}+1}{(\alpha u)^{q+1}}=\sigma_{\alpha,\beta,\gamma}\bigg(c_0\frac{u^{q+1}+1}{u^{q+1}}\bigg).$$
which shows that also the second equation of $\tilde{\mathcal{X}}_{a,b,n,s}$ is preserved.
Moreover $G$ is a group, as
$$
\sigma_{\alpha,\beta,\gamma}\circ\sigma_{\alpha^\prime,\beta^\prime,\gamma^\prime}=\sigma_{\alpha\alpha^\prime,\ \beta/{\alpha^\prime}^q + \alpha\beta^\prime,\ \gamma\gamma^\prime}.
$$
Then $G$ is a subgroup of $\aut(\tilde{\mathcal{X}}_{a,b,n,s})$.
In fact, it is not difficult to show that $G$ is inherited on $\tilde{\mathcal{X}}_{a,b,n,s}=\tilde{\mathcal{Y}}_{n,s}/\tilde{E}_{p^b}$ from the authomorphisms of $\tilde{\mathcal{Y}}_{n,s}$; in particular, $G\cong {\rm N}_H(\tilde{E}_{p^b})/\tilde{E}_{p^b}$, where ${\rm N}_H(\tilde{E}_{p^b})$ is the normalizer of $\tilde{E}_{p^b}$ in the group $H\cong {\rm SL}(2,q)\times C_{m/s}$ of Equation \eqref{eq:groupGH}.

The group $G$ maps $u$ to a multiple of $u$, and hence fixes its divisor, given in Equation \eqref{eq:xabns_div_rho} (recall that $u=\rho$); thus, $G$ acts on the set $\overline{O}_\infty:=\{\overline{P}_\infty^j \mid j=2,\ldots,q/p^b+1\}$ and fixes $\overline{P}_\infty^1$.

By Lemma \ref{highnormal}, $G=\overline{S}\rtimes \overline{C}$, where $\overline{S}$ is the Sylow $p$-subgroup of $G$ and $\overline{C}$ is cyclic of order coprime with $p$.
Clearly, $\sigma_{\alpha,\beta,\gamma}\in G$ is a $p$-element if and only if $\alpha=\gamma=1$, so that
$$ \overline{S}=\{\sigma_{1,\beta,1} \,:\, \beta \in \mathbb{F}_{q^2}, \ \beta^q+\beta=0 \}\leq G. $$
The map $\beta\mapsto\sigma_{1,\beta,1}$ is a homomorphism, and $\beta$ is in its kernel if and only if $c_0\beta\in\mathbb{F}_{p^b}$, i.e. $\beta=c_0^{-1}\mu$ for some $\mu\in\mathbb{F}_{p^b}$; thus, $\overline{S}$ is elementary abelian of size $q/p^b$.
Also, $\overline{S}$ acts transitively on $\overline{O}_{\infty}$ (see Corollary \ref{1actionp2}); hence, $\overline{O}_\infty$ is an orbit of $G$, and $\overline{C}$ fixes $\overline{O}_\infty$ pointwise.

From the definition of $G$ we have that $\overline{C}$ has order $(q+1)(p^b-1)m/s$; up to conjugacy in $G$, we can assume that $\overline{C}$ contains the scalar multiplications in $z$ :
$$
C_{m/s}:=\left\{ \sigma_{1,0,\gamma}\colon(u,v,z)\mapsto(u,v,\gamma z)\mid \gamma^{m/s}=1 \right\}\leq \overline{C}.
$$
Moreover, $G$ acts on $\overline{O}$, the set of places below $O$ in $\tilde{\mathcal{Y}}_{n,s}/\tilde{\mathcal{X}}_{a,b,n,s}$, which has size $(q^3-q)/p^b$.
In particular, it is easily seen that $C_{m/s}$ fixes $\overline{O}$ pointwise, and $G$ has $\frac{q-1}{p^b-1}$ many orbits $\overline{O}_i\subseteq\overline{O}$, $i=1,\ldots,\frac{q-1}{p^b-1}$, each of size $(q+1)(p^b-1)q/p^b$.
Finally, $\{\overline{P}_\infty^1\}$, $\overline{O}_\infty$, $\overline{O}_1,\ldots,\overline{O}_{(q-1)/(p^b-1)}$ are the only short orbits of $G$.

\begin{theorem} \label{fullXtilda}
If $m/s$ does not divide $q^2-q+1$, 
then $\aut(\tilde{\mathcal{X}}_{a,b,n,s})=G = \overline{S} \rtimes \overline{C}$.
\end{theorem}

\begin{proof}
Fix the notation $A:=\aut(\tilde{\mathcal{X}}_{a,b,n,s})$.
Note that ${\rm Fix}(A)$ is rational, being a subfield of ${\rm Fix}(\overline{S}\rtimes C_{m/s})=\mathbb{F}_{q^{2n}}(u)$.
By Theorem \ref{thm:semigroupsXaa_compare} and Propositions \ref{prop:Xab1}, \ref{prop:Xab2}, the hypothesis $m/s\nmid(q^2-q+1)$ implies that three points $\overline{P}_\infty^1$, $\overline{P}_\infty^j\in\overline{O}_\infty$ and $\overline{P}\in\overline{O}$ are in three pairwise distinct orbits under $A$.
We distinguish two cases, according to $A$ fixing $\overline{P}_\infty^1$ or not. 

\textbf{Case 1}: $A=A_{\overline{P}_\infty^1}$. By Lemma \ref{highnormal}, $A=S \rtimes C$ where $S$ is a $p$-group containing $\overline{S}$, and $C=\langle\sigma\rangle$ is a cyclic group of order prime to $p$ and divisible by $m(q+1)(p^b-1)/s$; up to conjugacy in $A$, we can assume that $\overline{C}\leq C$.

We first prove that $C=\overline{C}$.
The quotient group $C/C_{m/s}$ is an automorphism group of the fixed field of $C_{m/s}$, of size at least $|\overline{C}/C_{m/s}|=(q+1)(p^b-1)$. Moreover, ${\rm Fix}(C_{m/s})=\mathbb{F}_{q^{2n}}(u,v)$ where $ \sum_{i=0}^{f-1} v^{p^{ib}}=c_0\frac{u^{q+1}+1}{u^{q+1}}=c_0+c_0/u^{q+1}$. Note that $\mathbb{F}_{q^{2n}}(u,v)=\mathbb{F}_{q^{2n}}(U,v)$ with $\sum_{i=0}^{f-1} v^{p^{ib}}=c_0+c_0U^{q+1}$, where $U:=1/u$, that is, it is the function field of a plane curve given by separated polynomials as described in \cite[Chapter 12]{HKT}, and $b$ is the largest integer such that 
the polynomial $\sum v^{p^{ib}}$ in $v$ is $p^b$-linearized. 
Then, by \cite[Theorem 3.3]{BMZsep}, $|C/C_{m/s}|=(q+1)(p^b-1)$, that is, $C=\overline{C}$. 
We remark and amend here an erratum to \cite[Theorem 3.3]{BMZsep}: the final statement ``$B(X)=X^m$ up to an affine transformation in $X$" must be replaced by ``$B(X)=X^m$ up to a projective transformation in $X$". In fact, a projective transformation $x\mapsto\frac{ax+b}{cx+d}$ may be needed in the fourth line of the proof.

We now prove that $S=\overline{S}$.
Denote by $\mathcal{O}_\infty$ the $A$-orbit containing $\overline O_\infty$. 
By \cite[Lemma.129]{HKT}, $S$ acts with long orbits on $\mathcal{O}_\infty$.
Hence $|\mathcal{O}_\infty|= p^c q/p^b$, where $p^c=|S|/|\overline{S}|$.
If $\mathcal{O}_\infty=\overline{O}_\infty$ then $S=\overline{S}$ and the claim is proved.
Now suppose by contradiction that $\overline{O}_\infty \subsetneq \mathcal{O}_\infty$, i.e. $c\geq1$. 
As $G$ has no short orbits out of $\{\overline{P}_\infty^1\}\cup\overline{O}\cup\overline{O}_\infty$, we have $|\mathcal{O}_\infty| \geq |\overline{O}_\infty|+|G|$. By the orbit-stabilizer theorem,
$$
|A|=|\mathcal{O}_\infty|\cdot|A_{\overline{P}_\infty^2}|\geq\left(q/p^b+|G|\right)\cdot\frac{p^c|G|}{q/p^b}.
$$
By direct computation, this implies $|A|>18(g-1)$ for all possible values of $p^b$, $q$, $m/s\geq3$.
Then, by Theorem \ref{Hurwitz} $(i)$ and $(ii)$, $\aut(\tilde{\mathcal{X}}_{a,b,n,s})$ has at most three short orbits. Having also at least three short orbits as pointed out at the beginning of the proof, $A$ has exactly three short orbits, a non-tame one containing $\overline{P}_\infty^1$, the other two containing $\overline{O}_\infty$ and $\overline{O}$.
We now argue as in the proof of \cite[Theorem 11.56]{HKT}. Let $d_1$, $d_2$ and $d_3$ be the different exponent with respect to $A$ respectively at $P_1:=\overline{P}_\infty^1$, at $P_2\in\overline{O}_\infty$ and at $P_3\in\overline{O}$. Define $d_i^\prime=d_i/|A_{P_i}|$, $i=1,2,3$. Then the Hurwitz genus formula applied to $A$ reads
$$2g-2=-2|A|+|A|(d_1^\prime+d_2^\prime+d_3^\prime).$$
Since the orbit $\{\overline{P}_\infty^1\}$ is non-tame, $d_1>|A_{P_1}|-1$ and hence $d_1^\prime\geq1$. Since $C_{m/s}$ fixes both $P_2$ and $P_3$, we have $$d_i^\prime\geq\frac{|A_{P_i}|-1}{|A_{P_i}|}\geq\frac{m/s-1}{m/s}\geq\frac{3-1}{3}=\frac{2}{3}\quad\mbox{for }\;i=2,3. $$
Therefore
$$ |A|=\frac{2}{d_1^\prime+d_2^\prime+d_3^\prime-2}(g-1)\leq\frac{2}{1+2/3+2/3-2}(g-1)=6(g-1), $$
a contradiction to $|A|>18(g-1)$.

\textbf{Case 2}: $\aut(\tilde{\mathcal{X}}_{a,b,n,s}) \ne \aut(\tilde{\mathcal{X}}_{a,b,n,s})_{\overline{P}_\infty^1}$. Let $O_1$ be the $\aut(\tilde{\mathcal{X}}_{a,b,n,s})$-orbit containing $\overline{P}_\infty^1$, so that $|O_1|>1$.
By what we noted at the beginning of the proof, $O_1$ and $\overline{O}_\infty\cup\overline{O}$ intersect trivially; since $G$ fixes $\overline{P}_\infty^1$ and has no short orbits out of $\{\overline{P}_\infty^1\}\cup\overline{O}_\infty\cup\overline{O}$, this implies that $O_1\setminus\{P_\infty^1\}$ is the union of $k\geq1$ long orbits of $G$. Also, by Case 1 above, we have $\aut(\tilde{\mathcal{X}}_{a,b,n,s})_{\overline{P}_\infty^1}=G$. Then, by the orbit-stabilizer theorem,
$$
|\aut(\tilde{\mathcal{X}}_{a,b,n,s})|=|O_1|\cdot|\aut(\tilde{\mathcal{X}}_{a,b,n,s})_{\overline{P}_\infty^1}|=(1+k|G|)\cdot|G|.
$$

By direct computation, this yields $|\aut(\tilde{\mathcal{X}}_{a,b,n,s})|>44(g(\tilde{\mathcal{X}}_{a,b,n,s})-1)$ for all possible values of $p^b$, $q$, $k\geq1$ and $m/s\geq3$.
Then, by Theorem \ref{Hurwitz}, $A$ has at most three short orbits, which are indeed exactly three. With the very same arguments as in Case 1, we obtain $|A|\leq6(g-1)$, a contradiction to $|A|>44(g-1)$.
\end{proof}

The automorphism group of $\mathcal{X}_{a,b,n,s}$ has been computed in \cite{MTZ1}.
As automorphism groups are invariant under isomorphism, Theorem \ref{fullXtilda} can be combined with \cite[Theorem 1.2 and Corollary 3.23]{MTZ1} to obtain the following result. 

\begin{corollary}
Suppose that $3 \nmid n$ or $m/s \nmid q^2-q+1$. Then $\tilde{\mathcal{X}}_{a,b,n,s}$ and $\mathcal{X}_{a,b,n,s}$ are not isomorphic.
\end{corollary}

We are hence left with the case $m/s\mid(q^2-q+1)$, implying also $3\mid n$. We show that in this case the automorphism group is actually known, as $\tilde{\mathcal X}_{a,b,n,s}$ is isomorphic to ${\mathcal X}_{a,b,n,s}$.

\begin{lemma}
If $m/s\mid (q^2-q+1)$ (and hence $3\mid n$)
, then $\tilde{\mathcal{X}}_{a,b,n,s}:=\overline{\mathbb{F}}_{q^{2n}}(u,v,z)$ with
$$\tilde{\mathcal{X}}_{a,b,n,s}:\begin{cases} z^{m/s}=\frac{u^{q^2-1}-1}{u^{q-1}}, \\ {\rm Tr}_{\mathbb{F}_q/\mathbb{F}_{p^b}}(v)=c_0\frac{u^{q+1}+1}{u^{q+1}},\end{cases}$$
is isomorphic to $\mathcal{X}_{a,b,n,s}:=\overline{\mathbb{F}}_{q^{2n}}(X,Y,Z)$ with 
\begin{equation} \label{Xabns}
\mathcal{X}_{a,b,n,s}:\begin{cases}
    Z^{m/s}=Y^{q^2}-Y,\\
    c_0Y^{q+1}={\rm Tr}_{\mathbb{F}_q/\mathbb{F}_{p^b}}(X).
\end{cases}
\end{equation}
\end{lemma}

\begin{proof}
Let $t:=(q^2-q+1)/(m/s)$ and $\alpha,\beta\in\overline{\mathbb{F}}_{q^{2n}}$ such that ${\rm Tr}_{\mathbb{F}_q/\mathbb{F}_{p^b}}(\alpha)=-c_0$ and $\beta^{m/s}=-1$. Then, by direct computation, the mapping
$$ (u,v,z)\mapsto(X,Y,Z):=(v+\alpha,\,1/u,\,\lambda z/u^t) $$
is an isomorphism between $\tilde{\mathcal{X}}_{a,b,n,s}$ and $\mathcal{X}_{a,b,n,s}$.
\end{proof}

The following result now follows from \cite[Theorem 1.2]{MTZ1}.

\begin{proposition}
If $m/s$ divides $q^2-q+1$, then the automorphism group of $\mathcal{X}_{a,b,n,s}$ has order $\frac{q^3}{p^b}(q+1)(p^b-1)\frac{m}{s}$ and is isomorphic to a semidirect product $S_{q^3/p^b}\rtimes C_{(q+1)(p^b-1)m/s}$. 
\end{proposition}

\section{Galois covering problem for $\tilde{\mathcal{Y}}_{n,s}$ and $\tilde{\mathcal X}_{n,s,a,b}$}\label{sec:covering}

We consider the question whether the function fields $\tilde{\mathcal{Y}}_{n,s}$ and $\tilde{\mathcal X}_{n,s,a,b}$ are isomorphic to subfields of the Hermitian function field $\mathcal{H}_{q^n}$ over $\mathbb{F}_{q^{2n}}$. As for $m/s \mid (q^2-q+1)$ the function fields $\tilde{\mathcal{X}}_{a,b,n,s}$ and $\mathcal{X}_{a,b,n,s}$ are isomorphic (the same holds for $\tilde{\mathcal{Y}}_{n,s}$ and ${\mathcal{Y}}_{n,s}$), Theorems \ref{YtildaNotSubHerm} and \ref{XtildaNotSubHerm} are a direct consequence of \cite[Theorem 3.4]{TTT} and \cite[Theorem 4.4]{TTT}.

\begin{theorem} \label{YtildaNotSubHerm}
 If $s$ divides $q^2-q+1$ and $q>s(s+1)$, then the function field $\tilde{\mathcal Y}_{3,s}$ is not isomorphic to a subfield of the Hermitian function field over $\mathbb{F}_{q^6}$.  
\end{theorem}

\begin{theorem} \label{XtildaNotSubHerm}
    If $a \geq 2b+1$ then the function field $\tilde{\mathcal{X}}_{a,b,3,1}$ is not isomorphic to a subfield of the Hermitian function field over $\mathbb{F}_{q^6}$.
\end{theorem}

We now point out some further examples of function fields of type $\tilde{\mathcal{Y}}_{n,s}$ and $\tilde{\mathcal X}_{n,s,a,b}$ not isomorphic to Galois subfields of the Hermitian function field $\mathcal{H}_{q^n}$ can be found, i.e. for further values of $a,b,n$ and $s$ not covered by Theorems \ref{YtildaNotSubHerm} and \ref{XtildaNotSubHerm}.
Understanding whether in general $\tilde{\mathcal{Y}}_{n,s}$ and $\tilde{\mathcal X}_{n,s,a,b}$ are isomorphic to subfields of $\mathcal{H}_{q^n}$ remains an open problem.

Suppose $\tilde{\mathcal Y}_{n,s}$ (resp. $\tilde{\mathcal X}_{n,s,a,b}$) is isomorphic to a subfield $F \subseteq \mathcal{H}_{q^n}$ and let $d := [\mathcal{H}_{q^n} \, : \, F]$. Clearly $g:=g(F)=g(\tilde{\mathcal Y}_{n,s})$ (resp. $\tilde{\mathcal X}_{n,s,a,b}$). Using that each rational place of $\mathcal{H}_{q^n}$ lies over a rational place of $F$, together with the Hurwitz genus formula, one gets that
    \begin{equation}\label{eq:range}
\frac{q^{3n} + 1}{q^{2n}+1+2gq^n} \leq d \leq \frac{2g(\mathcal{H}_{q^n})-2}{2g-2}.
    \end{equation}
By substituting in \eqref{eq:range} the expression of $g$ (computed in Equation \eqref{eq:genusYns} and Lemma \ref{genXtilda}), one obtains a range for the possible values of $d$.
Theorems \ref{YtildaNotSubHerm} and \ref{XtildaNotSubHerm} were in fact obtained by having an empty range for $d$ in \eqref{eq:range}.
In the following some further examples are provided.

\begin{example}
Let $q=5$, $n=3$, $s=3$. Then 
Theorem \ref{YtildaNotSubHerm} does not apply to $\tilde{\mathcal Y}_{3,3}$. Assume by contradiction that $\tilde{\mathcal Y}_{3,3}$ is a Galois subfield of $\mathcal{H}_{5^3}$, with Galois group $G$. 
The bound \eqref{eq:range} gives $d=|G| \in \{16,17\}$.
As $d$ divides $|PGU(3,5^3)|$, we have $d=16$. 
By \cite[Theorem 2.4]{MZcomm}, in $PGU(3,5^3)$ elements of order $16$ do not exist, elements of order $2$ contribute $q^n+1$ to the different divisor, and elements of order $4$ or $8$ contribute $2$.
By the Hurwitz genus formula, the number $k\in\{1,\ldots, 15\}$ of involutions in $G$ satisfies $q^n(q^n-1)-2=16(2g-2)+k(q^n+1)+(15-k)2$; 
since $g(\tilde{\mathcal Y}_{3,3})=442$, we get $1356=124k$, a contradiction.
\end{example}

\begin{example}
Let $p=3$, $a=2$, $b=1$, $n=3$, $s=1$. Then 
Theorem \ref{XtildaNotSubHerm} does not apply to $\tilde{\mathcal X}_{2,1,3,1}$. Suppose that $\tilde{\mathcal X}_{2,1,3,1}$ is a Galois subfield of $\mathcal{H}_{9^3}$, with Galois group $G$.
The bound \eqref{eq:range} gives $d=|G|=28$.
By \cite[Theorem 2.4]{MZcomm}, nontrivial elements of $G$ contribute to the different divisor either $q^n+1$, if they have order $2$, or $2$, otherwise. By the Hurwitz genus formula, there exists $k\in\{1,\ldots, 27\}$ 
such that
$q^n(q^n-1)-2=28(2g-2)+k(q^n+1)+(27-k)2$;
since $g(\tilde{\mathcal X}_{2,1,3,1})=9369$, we get
$6048=728k$, a contradiction.
\end{example}

Also for $n>3$ function fields not isomorphic to subfields of $\mathcal{H}_{q^n}$ can be found.


\begin{example}
Let $p=2$, $a=2$, $b=1$, $n=5$, $s=1$. If $\tilde{\mathcal{X}}_{5,1,2,1}$ is a Galois subcover of $\mathcal{H}_{4^5}$ with Galois group $G$, then the bound \eqref{eq:range} yields $d=|G| \in \{130,132,143\}$.

Suppose that $d=130$. Then the degree $\Delta$ of the different divisor is $119090$. Nontrivial elements in $G$ have order $2, 5, 10, 13, 26, 65$ or $130$ and there exist at least one element of order $2$ and $12$ elements of order $13$. 
    By \cite[Theorem 2.4]{MZcomm}, elements of order $2$ contribute $q^n+2$ to $\Delta$, elements of order $13$ contribute $3$, and the remaining $129-13$ elements contribute at most $q^n+1$. Since $q^n+2+12\cdot 3+(129-14)(q^n+1)<\Delta<q^n+2+12\cdot 3+(129-13)(q^n+1)$, 
    at most one element of $G$ gives a contribution $\gamma < q^n+1$. Going through all the possibilities for $\gamma$ in \cite[Theorem 2.4]{MZcomm} shows that equality can never be reached.

Suppose that $d=132$. Then $d=104806$; hence, 
by \cite[Theorem 2.4]{MZcomm}, elements of order $2$ contribute $q^n+2$ to $\Delta$ while all the other nontrivial elements contribute $2$. Hence there exists $k\in\{1,\ldots,131\}$ with $104806=(q^n+2)k+2(131-k)=1024k+262$, a contradiction.
   
Suppose that $d=143$. Then $d=26244$. 
By \cite[Theorem 2.4]{MZcomm}, elements of order $143$ do not exist in $PGU(3,q^n)$, elements of order $11$ contribute $2$ to $\Delta$, while elements of order $13$ contribute $3$. Thus there exists $k\in\{1,\ldots,142\}$ with $26244=2k+3(142-k)$, a contradiction.
\end{example}

\section*{Acknowledgments}
This work was supported by a research grant (VIL”52303”) from Villum Fonden.

    \end{document}